\documentclass[a4paper,american]{amsart}
\usepackage{mathptmx}
\usepackage[T1]{fontenc}
\usepackage[latin9]{inputenc}
\usepackage{babel}
\usepackage{units}
\usepackage{amsmath}
\usepackage{amsthm}
\usepackage{amssymb}
\usepackage{wasysym}
\usepackage[unicode=true,pdfusetitle,
 bookmarks=true,bookmarksnumbered=false,bookmarksopen=false,
 breaklinks=false,pdfborder={0 0 0},pdfborderstyle={},backref=false,colorlinks=true]
 {hyperref}
\hypersetup{
 bookmarksopen}

\makeatletter

\pdfpageheight\paperheight
\pdfpagewidth\paperwidth

\theoremstyle{plain}
\numberwithin{equation}{section}
\numberwithin{figure}{section}
\numberwithin{table}{section}
\newtheorem{stthm}{\protect\theoremname}[section]
\newtheorem{spprop}{\protect\propositionname}[section]

\usepackage{enumitem}
\theoremstyle{remark}
\newtheorem*{rem*}{\protect\remarkname}
\theoremstyle{plain}
\newtheorem{sllem}{\protect\lemmaname}[section]
\theoremstyle{remark}
\newtheorem*{claim*}{\protect\claimname}

\DeclareMathOperator*{\BCo}{\bigcirc}

\makeatother

\providecommand{\claimname}{Claim}
\providecommand{\lemmaname}{Lemma}
\providecommand{\propositionname}{Proposition}
\providecommand{\remarkname}{Remark}
\providecommand{\theoremname}{Theorem}

\begin{document}
\def\rightmark{ON WEIGHTED $L^p$-SOBOLEV ESTIMATES FOR SOLUTIONS OF THE $\bar\partial$ EQUATION}
\def\leftmark{P. CHARPENTIER \& Y. DUPAIN}
\def\RSsectxt{Section~}%
\title{On Weighted $L^{p}$-Sobolev Estimates for Solutions of the $\overline{\partial}$-equation
on Lineally Convex Domains of Finite Type and Application}
\author{P. Charpentier \& Y. Dupain}
\begin{abstract}
We obtain some weighted $L^{p}$-Sobolev estimates with gain on $p$
and the weight for solutions of the $\overline{\partial}$-equation
in lineally convex domains of finite type in $\mathbb{C}^{n}$ and
apply them to obtain weighted $L^{p}$-Sobolev estimates for weighted
Bergman projections of convex domains of finite type for quite general
weights equivalent to a power of the distance to the boundary.
\end{abstract}

\address{P. Philippe Charpentier, Univ. Bordeaux, CNRS, Bordeaux INP, IMB,
UMR 5251, F-33400 Talence, France}
\email{Philippe Charpentier: philippe.charpentier@math.u-bordeaux.fr}
\maketitle

\section{Introduction}

The motivation of this paper is to extend the weighted $L^{p}$-Sobolev
estimates for weighted Bergman projections obtained in \cite[(1) of Theorem 1.1]{CDM}
for convex domains of finite type in $\mathbb{C}^{n}$ to more general
weights.

For this we use the method of Section 4 of \cite{ChDu18} which needs
weighted $L^{p}$-Sobolev estimates for solutions of the $\overline{\partial}$-equation
with suitable gains, the weight being a power of the distance to the
boundary. So in this paper we try to obtain such estimates.

Estimates for derivatives of solutions of the $\overline{\partial}$-equation
in convex domains of finite type has already a long history. The first
result was obtained by W. Alexandre for $\mathcal{C}^{k}$ estimates
in \cite{Ale05,Ale06}. For $L^{p}$-Sobolev estimates, with gain
on the Sobolev index, partial results were obtained by Ahn H. and
Cho H. in \cite{Ahn_H-Cho_H-Opt-Sob-cvx2003} and complete results
announced by Yao L. in \cite{yao-liding-Sob-Hol-d-bar-cvx-finite-multitype-2022}.

Here we obtain partial weighted $L^{p}$-Sobolev estimates, with both
gain on $p$ and the power of the weight, for solutions of the $\overline{\partial}$-equation
on the more general case of lineally convex domains of finite type.
We use the operator constructed and used in \cite{CDMb,ChDu18} and
we apply them to weighted Bergman projections.

The paper is organized as follows. In Section \ref{sec:Notations-and-main-results}
we fix our notations and state the theorems about the estimates on
solutions of the $\overline{\partial}$-equation. In Section \ref{sec:Geometry-of-lineally-convex}
we give the main properties of the geometry of lineally convex domains
of finite type, in Section \ref{sec:Definition-and-properties-solution-d-bar}
we recall the construction of the operator solving the $\overline{\partial}$-equation,
and, in Section \ref{sec:Proof-of-Theorems} we give the proofs of
the Theorems stated in Section \ref{sec:Notations-and-main-results}.
Finally, in Section \ref{sec:Application-to-weighted-Bergman} we
give the estimates for weighted Bergman projections we can obtain
with our method.

\section{\label{sec:Notations-and-main-results}Notations and main results}

Throughout this paper $\Omega=\left\{ \rho<0\right\} $ is a smoothly
bounded lineally convex domain of finite type $\leq m$ in $\mathbb{C}^{n}$
which means that (c.f. \cite{CDMb}) for all point $p$ in the boundary
$\partial\Omega$ of $\Omega$, there exists a neighborhood $W$ of
$p$ such that, for all point $z\in\partial\Omega\cap W$,
\[
\left(z+T_{z}^{1,0}\right)\cap(\Omega\cap W)=\emptyset,
\]
where $T_{z}^{1,0}$ is the holomorphic tangent space to $\partial\Omega$
at the point $z$ and that there exists a smooth defining function
$\rho$ of $\Omega$ such that, for $\delta_{0}$ sufficiently small,
the domains $\Omega_{t}=\left\{ \rho(z)<t\right\} $, $-\delta_{0}\leq t\leq\delta_{0}$,
satisfy the same condition.

\medskip{}

For $d$ a non negative integer and $p\in\left[1,+\infty\right[$,
$\delta_{\Omega}$ being the distance to the boundary of $\Omega$,
we denote by $L_{d}^{p}\left(\Omega,\delta_{\Omega}^{\gamma}\right)$,
$-1<\gamma$ , the space of functions $f$ such that, for any derivative
$D$ of order $\alpha$, $0\leq\alpha\leq d$, $Df$ belongs to $L^{p}\left(\Omega,\delta_{\Omega}^{\gamma}\right)$
and by $L_{d,0}^{p}\left(\Omega,\delta_{\Omega}^{\gamma}\right)$
the closure in $L_{d}^{p}\left(\Omega,\delta_{\Omega}^{\gamma}\right)$
of the set of $\mathcal{C^{\infty}}$ functions with compact support
in $\Omega$. Moreover we denote by $L_{d,r}^{p}\left(\Omega,\delta_{\Omega}^{\gamma}\right)$
or simply $L_{d,r}^{p}\left(\delta^{\gamma}\right)$ the space of
forms of degree $r$ whose coefficients belong to $L_{d}^{p}\left(\Omega,\delta_{\Omega}^{\gamma}\right)$
and $L_{d,0,r}^{p}\left(\Omega,\delta_{\Omega}^{\gamma}\right)$ or
simply $L_{d,0,r}^{p}\left(\delta^{\gamma}\right)$ the space of forms
of degree $r$ whose coefficients belong to $L_{d,0}^{p}\left(\Omega,\delta_{\Omega}^{\gamma}\right)$.
\begin{stthm}
\label{thm:Theroem-without-derivative}Let $\Omega=\left\{ \rho<0\right\} $
be a smoothly bounded lineally convex domain of finite type $\leq m$.
Let $p\in\left[1,+\infty\right[$. Let $\gamma$ and $\gamma'$ be
two real numbers such that $\gamma'>-1$ and $\gamma-\nicefrac{p}{m}\leq\gamma'\leq\gamma$.
Then there exists a linear operator $T$, depending on $\rho$ and
$\gamma$, such that, for any $\overline{\partial}$-closed $\left(0,r\right)$-form
($1\leq r\leq n-1$) whose coefficients are in $L^{p}\left(\Omega,\delta_{\Omega}^{\gamma}\right)$,
$Tf$ is a solution of the equation $\overline{\partial}(Tf)=f$ such
that:
\begin{enumerate}
\item For $1\leq p<m(\gamma+n)+2-(m-2)(r-1)$,
\[
\left\Vert Tf\right\Vert _{L^{q}\left(\delta_{\Omega}^{\gamma'}\right)}\lesssim\left\Vert f\right\Vert _{L^{p}\left(\delta_{\Omega}^{\gamma}\right)},
\]
where $p\text{, }\gamma\text{, }q\text{, and }\gamma'$ satisfy the
following condition:\\
if $p=1$ and $\gamma\geq0$, $\frac{1}{q}>\frac{1}{p}-\frac{1-\frac{m}{p}(\gamma-\gamma')}{m(\gamma'+n)+2-(m-2)(r-1)}$
and $\frac{1}{q}=\frac{1}{p}-\frac{1-\frac{m}{p}(\gamma-\gamma')}{m(\gamma'+n)+2-(m-2)(r-1)}$
otherwise.
\item For $p>m(\gamma+n)+2-(m-2)(r-1)$,
\[
\left\Vert Tf\right\Vert _{\Lambda_{\alpha}}\apprle\left\Vert f\right\Vert _{L^{p}\left(\delta_{\Omega}^{\gamma}\right)},
\]
with $\alpha=\frac{1}{m}\left[1-\frac{m(\gamma+n)+2-(m-2)(r-1)}{p}\right]$.
\end{enumerate}
\end{stthm}

\begin{rem*}
For $\gamma=\gamma'$ this statement is identical to \cite[Theorem 2.1]{ChDu18}
but (1) is stronger when $\gamma\neq\gamma'$. For $p=m(\gamma+n)+2-(m-2)(r-1)$
it could be proved that $Tf\in BMO(\Omega)$ but we will not do it
here.
\end{rem*}

\begin{stthm}
\label{thm:Theorem-Sobolev-Lp-k}Let $\Omega=\left\{ \rho<0\right\} $
be a smoothly bounded lineally convex domain of finite type $\leq m$.
Let $p\in\left[1,+\infty\right[$. Let $\gamma$ and $\gamma'$ be
two real numbers such that $\gamma'>-1$ and $\gamma-\nicefrac{p}{m}\leq\gamma'\leq\gamma$.
Then there exists a linear operator $T$, depending on $\rho$ and
$\gamma$, such that, for any $\overline{\partial}$-closed $\left(0,r\right)$-form
whose coefficients are in $L_{1}^{p}\left(\Omega,\delta_{\Omega}^{\gamma}\right)$,
$Tf$ is a solution of the equation $\overline{\partial}(Tf)=f$ such
that:
\[
\left\Vert Tf\right\Vert _{L_{1,r-1}^{q}\left(\delta_{\Omega}^{\gamma'}\right)}\lesssim\left\Vert f\right\Vert _{L_{1,r}^{p}\left(\delta_{\Omega}^{\gamma}\right)},
\]
where $p\text{, }\gamma\text{, }q\text{, and }\gamma'$ satisfy the
following condition: $\frac{1}{q}=\frac{1}{p}-\alpha$ with:
\begin{enumerate}
\item If $\gamma\geq p-1$: $\alpha<\frac{1-\frac{m}{p}(\gamma-\gamma')}{m(\gamma'+n)+2-(m-2)(r-1)}$;
\item If $\gamma<p-1$: $\gamma'\geq\gamma-\frac{\gamma+1}{p}$ and $\alpha<\min\left\{ \frac{1-\frac{m}{p}(\gamma-\gamma')}{m(\gamma'+n)+2-(m-2)(r-1)},\frac{\frac{\gamma+1}{p}-(\gamma-\gamma')}{2n+\gamma'}\right\} $.
\end{enumerate}
\end{stthm}

\begin{stthm}
\label{thm:Theorem-Sobolev-Lp-k-0-compact-supp}Let $\Omega$ be a
smoothly bounded lineally convex domain of finite type $\leq m$.
Let $p\in\left[1,+\infty\right[$. Let $d\geq2$ be an integer. Let
$\gamma$ and $\gamma'$ be two real numbers such that $\gamma'>-1$
and $\gamma-\nicefrac{p}{m}\leq\gamma'\leq\gamma$. Then there exists
a linear operator $T$, depending on $\rho$ and $\gamma$, such that,
for any $\overline{\partial}$-closed $\left(0,r\right)$-form ($1\leq r\leq n-1$)
$f$ whose coefficients are in $L_{d}^{p}\left(\Omega,\delta_{\Omega}^{\gamma}\right)$
and $p\in\left[1,+\infty\right[$, $Tf$ is a solution of the equation
$\overline{\partial}(Tf)=f$ such that:
\[
\left\Vert Tf\right\Vert _{L_{d,r-1}^{q}\left(\delta_{\Omega}^{\gamma'}\right)}\lesssim\left\Vert f\right\Vert _{L_{d,r}^{p}\left(\delta_{\Omega}^{\gamma}\right)},
\]
if one of the three following conditions is satisfied:
\begin{enumerate}
\item $f\in L_{d,0,r}^{p}\left(\Omega,\delta_{\Omega}^{\gamma}\right)$,
\item $f\in L_{d,r}^{p}\left(\Omega,\delta_{\Omega}^{\gamma}\right)$ and
$\gamma\geq dp-1$,
\item $f\in L_{d,r}^{p}\left(\Omega,\delta_{\Omega}^{\gamma}\right)$ and
$d(p-1)\leq\gamma<dp-1$,
\end{enumerate}
with $p\text{, }\gamma\text{, }q\text{, and }\gamma'$ satisfy:
\begin{itemize}
\item when condition (1) or (2) is verified, $\frac{1}{q}>\frac{1}{p}-\frac{1-\frac{m}{p}(\gamma-\gamma')}{m(\gamma'+n)+2-(m-2)(r-1)}$,
\item when condition (3) is verified, $\frac{1}{q}>\frac{1}{p}-\min\left\{ \frac{1-\frac{m}{p}(\gamma-\gamma')}{m(\gamma'+n)+2-(m-2)(r-1)},\frac{\frac{\gamma+1}{p}-(\gamma-\gamma')}{2n+\gamma'}\right\} $.
\end{itemize}
\end{stthm}

\section{\label{sec:Geometry-of-lineally-convex}Geometry of lineally convex
domains of finite type}

To simplify the reading of this paper we recall here the main properties
of the geometry of lineally convex domains of finite type taken from
\cite{Conrad_lineally_convex}, \cite{CDMb} and \cite{ChDu18}.

For $\zeta$ close to $\partial\Omega$ and $\varepsilon\leq\varepsilon_{0}$,
$\varepsilon_{0}$ small, define, for all unitary vector $v$,
\begin{equation}
\tau\left(\zeta,v,\varepsilon\right)=\sup\left\{ c\mbox{ such that }\rho\left(\zeta+\lambda v\right)-\rho(\zeta)<\varepsilon,\,\forall\lambda\in\mathbb{C},\,\left|\lambda\right|<c\right\} .\label{eq:definition-tau-general}
\end{equation}
Note that the lineal convexity hypothesis implies that the function
$\left(\zeta,v,\varepsilon\right)\mapsto\tau(\zeta,v,\varepsilon)$
is smooth. In particular, $\zeta\mapsto\tau(\zeta,v,\delta_{\Omega}(\zeta))$
is a smooth function. The pseudo-balls $B_{\varepsilon}(\zeta)$ (for
$\zeta$ close to the boundary of $\Omega$) of the homogeneous space
associated to the anisotropic geometry of $\Omega$ are
\begin{equation}
B_{\varepsilon}(\zeta)=\left\{ \xi=\zeta+\lambda u\mbox{ with }\left|u\right|=1\mbox{ and }\left|\lambda\right|<c_{0}\tau(\zeta,u,\varepsilon)\right\} \label{eq:def-pseudo-balls}
\end{equation}
where $c_{0}$ is chosen sufficiently small depending only on the
defining function $\rho$ of $\Omega$.

Let $\zeta$ and $\varepsilon$ be fixed. Then, an orthonormal basis
$\left(v_{1},v_{2},\ldots,v_{n}\right)$ is called \emph{$\left(\zeta,\varepsilon\right)$-extremal}
(or $\varepsilon$-\emph{extremal}, or simply \emph{extremal}) if
$v_{1}$ is the complex normal (to $\rho$) at $\zeta$, and, for
$i>1$, $v_{i}$ belongs to the orthogonal space of the vector space
generated by $\left(v_{1},\ldots,v_{i-1}\right)$ and minimizes $\tau\left(\zeta,v,\varepsilon\right)$
in the unit sphere of that space. In association to an extremal basis,
we denote
\begin{equation}
\tau(\zeta,v_{i},\varepsilon)=\tau_{i}(\zeta,\varepsilon).\label{eq:definition-tau-extremal-basis}
\end{equation}
so $i\rightarrow\tau_{i}(\zeta,\varepsilon)$ is increasing. Moreover
the finite type condition implies that
\begin{equation}
\tau_{1}(\zeta,\varepsilon)=\varepsilon\mathrm{\,\ and\,\ }\varepsilon^{1/2}\apprle\tau_{i}(\zeta,\varepsilon)\apprle\varepsilon^{1/m},\,i\geq1.\label{eq:estimation-tau-i-1-n}
\end{equation}

If $v$ is any unit vector and $\lambda\geq1$,
\[
\lambda^{\nicefrac{1}{m}}\tau(\zeta,v,\varepsilon)\leq\frac{1}{C}\tau(\zeta,v,\lambda\varepsilon)\leq C\lambda\tau(\zeta,v,\varepsilon),
\]
where $m$ is the type of $\Omega$.

Then we define the polydiscs $AP_{\varepsilon}(\zeta)$ by
\begin{equation}
AP_{\varepsilon}(\zeta)=AP(\zeta,\varepsilon)=\left\{ z=\zeta+\sum_{k=1}^{n}\lambda_{k}v_{k}\mbox{ such that }\left|\lambda_{k}\right|\leq c_{0}A\tau_{k}(\zeta,\varepsilon)\right\} .\label{eq:def-polydisk}
\end{equation}

$P_{\varepsilon}(\zeta)=P(\zeta,\varepsilon)$ being the corresponding
polydisc with $A=1$. We choose $c_{0}<1$ so that $z\in P_{\varepsilon}(\zeta)$
implies
\begin{equation}
\left|\delta_{\Omega}(z)-\delta_{\Omega}(\zeta)\right|\leq\varepsilon/2.\label{eq:dist-boundary-in-polydisc}
\end{equation}

\begin{rem*}
Note that there is neither unicity of the extremal basis $\left(v_{1},v_{2},\ldots,v_{n}\right)$
nor of associated polydisk $P_{\varepsilon}(\zeta)$. However the
polydisks associated to two different $\left(\zeta,\varepsilon\right)$-extremal
basis are equivalent. Thus in all the paper $P_{\varepsilon}(\zeta)$
will denote a polydisk associated to any $\left(\zeta,\varepsilon\right)$-extremal
basis and $\tau_{i}(\zeta,\varepsilon)$ the radius of $P_{\varepsilon}(\zeta)$.
\end{rem*}

Moreover there exists $C>0$ such that for $\zeta$ close to $\partial\Omega$
and $\varepsilon>0$ small (see \cite{Conrad_lineally_convex} and
\cite{CDMb}):

If $v$ is a unit vector then:

\begin{enumerate}
\item $z=\zeta+\lambda v\in P_{\varepsilon}(\zeta)$ implies $\left|\lambda\right|\leq C\tau(\zeta,v,\varepsilon)$,
\item $z=\zeta+\lambda v$ with $\left|\lambda\right|\leq\tau(\zeta,v,\varepsilon)$
implies $z\in CP_{\varepsilon}(\zeta)$.
\end{enumerate}
\medskip{}

\begin{sllem}
\label{lem:3.4-maj-deriv-rho-equiv-tho-i-z-zeta}For $z$ close to
$\partial\Omega$, $\varepsilon$ small and $\zeta\in P_{\varepsilon}(z)$
or $z\in P_{\varepsilon}(\zeta)$, we have, for all $1\leq i\leq n$:

\begin{enumerate}
\item $\tau_{i}(z,\varepsilon)=\tau\left(z,v_{i}\left(z,\varepsilon\right),\varepsilon\right)\simeq\tau\left(\zeta,v_{i}\left(z,\varepsilon\right),\varepsilon\right)$
where $\left(v_{i}\left(z,\varepsilon\right)\right)_{i}$ is a $\left(z,\varepsilon\right)$-extremal
basis;
\item $\tau_{i}(\zeta,\varepsilon)\simeq\tau_{i}(z,\varepsilon)$;
\item In the coordinate system $\left(z_{i}\right)$ associated to the $\left(z,\varepsilon\right)$-extremal
basis, $\left|\frac{\partial\rho}{\partial z_{i}}(\zeta)\right|\lesssim\frac{\varepsilon}{\tau_{i}}$
where $\tau_{i}$ is either $\tau_{i}\left(z,\varepsilon\right)$
or $\tau_{i}\left(\zeta,\varepsilon\right)$.
\end{enumerate}
\end{sllem}

\begin{rem*}
In (1) above $\tau\left(\zeta,v_{i}\left(z,\varepsilon\right),\varepsilon\right)$
is not $\tau_{i}\left(\zeta,\varepsilon\right)$ because the extremal
basis at $z$ and $\zeta$ are different but (2) implies that these
quantities are equivalent.
\end{rem*}

We also define
\[
d(\zeta,z)=\inf\left\{ \varepsilon\mbox{ such that }z\in P_{\varepsilon}(\zeta)\right\} .
\]

The fundamental result here is that $d$ is a pseudo-distance.

Moreover the pseudo-balls $B_{\varepsilon}$ and the polydiscs $P_{\varepsilon}$
are equivalent in the sense that there exists a constant $K>0$ depending
only on $\Omega$ and $\rho$ such that
\begin{equation}
\frac{1}{K}P_{\varepsilon}(\zeta)\subset B_{\varepsilon}(\zeta)\subset KP_{\varepsilon}(\zeta)\label{eq:equivalence-pseudoballs-polydisk}
\end{equation}

so

\[
d(\zeta,z)\simeq\inf\left\{ \varepsilon\mbox{ such that }z\in B_{\varepsilon}(\zeta)\right\} 
\]
and 
\begin{equation}
d(\zeta,z)\simeq d(z,\zeta).\label{eq:symetry-pseudo-distance}
\end{equation}

Moreover 
\begin{equation}
\delta(z)+d(z,\zeta)\simeq\delta(\zeta)+d(z,\zeta).\label{eq:delta+pseudo-dist-symetric}
\end{equation}

\begin{sllem}
\label{lem:Lemma-5-1} For $z\in\Omega$, close to $\partial\Omega$,
$\delta$ small, $k=0,\ldots,n-1$ and $-1<\mu<1$,
\begin{equation}
\int_{P_{\delta}(z)}\frac{d\lambda(\zeta)}{\left|z-\zeta\right|^{2(n-k)-1+\mu}}\apprle_{\mu}\tau_{k+1}(z,\delta)^{1-\mu}\prod_{j=1}^{k}\tau_{j}^{2}(z,\delta).\label{eq:integral-mod(z-zeta)-inverse-power-1+mu}
\end{equation}

For $0<\mu<2$, $k=0,\ldots,n-1$, $\mathcal{T}$ being the real tangent
space to $\rho$ at the point $z$
\begin{multline}
\int_{P_{\delta}(z)\cap\mathcal{T}}\frac{d\sigma(\zeta)}{\left|z-\zeta\right|^{2(n-k-1)+\mu}}\apprle_{\mu}\\
\left\{ \begin{array}{cc}
\tau_{1}(z,\delta)\tau_{k+1}(z,\delta)^{2-\mu}\prod_{j=2}^{k}\tau_{j}(z,\delta)^{2} & \text{if }k\geq2\\
\tau_{1}(z,\delta)\tau_{2}(z,\delta)^{2-\mu} & \text{if }k=1\\
\tau_{1}(z,\delta)^{1-\mu} & \text{if }k=0\text{ and }0<\mu<1.
\end{array}\right.\label{eq:integral-mod(z-zeta)-inverse-power-1+mu-tangent}
\end{multline}

For $-1<\beta\leq0$, $k=1,\ldots,n-1$, $0\leq\mu<1$ and $\delta\geq\delta_{\Omega}(\zeta)$
\begin{equation}
\int_{P_{\delta}(\zeta)}\delta_{\Omega}(z)^{\beta}\frac{d\lambda(z)}{\left(\left|z-\zeta\right|^{2(n-k)-1+\mu}\right)}\lesssim_{\beta,\mu}\delta^{\beta}\tau_{k+1}\left(\zeta,\delta\right)^{1-\mu}\prod_{j=1}^{k}\tau_{j}^{2}\left(\zeta,\delta\right).\label{eq:3.16}
\end{equation}
\end{sllem}

\begin{proof}
We use the coordinate $\xi=\left(\xi_{i}\right)_{1\leq i\leq n}$
system associated to a $\left(z,\delta\right)$-extremal basis and
we denote $\xi'=\left(\xi_{k+1},\ldots,\xi_{n}\right)$ and, if $k<n-1$,
$\xi''=\left(\xi_{k+2},\ldots,\xi_{n}\right)$.
\begin{itemize}
\item \emph{Proof of \eqref{eq:integral-mod(z-zeta)-inverse-power-1+mu}}:
Let
\[
E_{0}=\left\{ \zeta\in P_{\delta}(z)\mathrm{\,such\,that\:}\left|z'-\zeta'\right|\leq2\tau_{k+1}(z,\zeta)\right\} 
\]
and
\[
E_{j}=\left\{ \begin{array}{cc}
\left\{ \zeta\in P_{\delta}(z)\mathrm{\,such\,that\,}\left|z''-\zeta''\right|\in\left[2^{j-1}\tau_{k+1}(z,\delta),2^{j}\tau_{k+1}(z,\delta)\right[\right\} , & \mathrm{if\,}k<n-1\\
\emptyset, & \mathrm{if\,}k=n-1.
\end{array}\right.
\]
Note that (see \eqref{eq:dist-boundary-in-polydisc}) $P_{\delta}(z)=\cup_{j\geq0}E_{j}$.\\
Then
\[
\int_{E_{0}}\frac{d\lambda(\zeta)}{\left|z-\zeta\right|^{2(n-k)-1+\mu}}\apprle_{\mu}\tau_{k+1}(z,\delta)^{1-\mu}\prod_{j=1}^{k}\tau_{j}(z,\delta)^{2}
\]
and, if $k<n-1$,
\[
\int_{E_{j}}\frac{d\lambda(\zeta)}{\left|z-\zeta\right|^{2(n-k)-1+\mu}}\apprle_{\mu}\left(2^{j}\tau_{k+1}(z,\delta)\right)^{-1-\mu}\prod_{j=1}^{k+1}\tau_{j}(z,\delta)^{2},
\]
which proves \emph{\eqref{eq:integral-mod(z-zeta)-inverse-power-1+mu}.}
\item \emph{Proof of \eqref{eq:integral-mod(z-zeta)-inverse-power-1+mu-tangent}}:
analog to the proof of \eqref{eq:integral-mod(z-zeta)-inverse-power-1+mu}.
\item \emph{Proof of \eqref{eq:3.16}}: Let $\mathcal{T}$ be the real tangent
space to $\rho$ at $\zeta$ and $\nu$ the inward real normal to
$\rho$ at $\zeta$.\\
For $z\in P_{\delta}(\zeta)$, let us write $z=Z(z)+t(z)\nu$, $Z(z)\in\mathcal{T}$,
and let $t_{0}$ such that $Z(z)+t_{0}\nu=W(z)\in\partial\Omega$
so $z=W(z)+t\nu$ with $t=t(z)-t_{0}$. Then $\delta_{\Omega}(z)\simeq t$
and $\left|z-\zeta\right|\simeq\left|Z(z)-\zeta\right|+\left|t(z)\right|$
and, by hypothesis on $\delta$ and \eqref{eq:dist-boundary-in-polydisc},
$\delta_{\Omega}(z)\leq2\delta$, thus
\begin{multline*}
\int_{P_{\delta}(\zeta)}\delta_{\Omega}(z)^{\beta}\frac{d\lambda(z)}{\left(\left|z-\zeta\right|^{2(n-k)-1+\mu}\right)}\\
\apprle\int_{P_{\delta}(\zeta)\cap\mathcal{T}}\left(\int_{0}^{2\delta}\frac{t^{\beta}dt}{\left|W-t\nu-\zeta\right|^{2(n-k)-1+\mu}}\right)d\sigma\\
\apprle\int_{P_{\delta}(\zeta)\cap\mathcal{T}}\frac{\delta^{\beta+1}}{\left|Z-\zeta\right|^{2(n-k-1)+1+\mu)}}d\sigma(Z)\\
\apprle\begin{cases}
\delta^{\beta+1}\delta\tau_{k+1}(\zeta,\delta)^{2-(1+\mu)}\prod_{j=2}^{k}\tau_{j}(\zeta,\delta)^{2} & \mathrm{if}\,k\geq1\\
\delta^{\beta}\tau_{k+1}(\zeta,\delta)^{1-\mu} & \mathrm{if}\,k=0
\end{cases}\,\mathrm{by\,\eqref{eq:integral-mod(z-zeta)-inverse-power-1+mu-tangent}}\\
=\begin{cases}
\delta^{\beta}\tau_{k+1}(\zeta,\delta)^{1-\mu}\prod_{j=1}^{k}\tau_{j}(\zeta,\delta) & \mathrm{if}\,k\geq1,\\
\delta^{\beta}\tau_{k+1}(\zeta,\delta)^{1-\mu} & \mathrm{if}\,k=0,
\end{cases}
\end{multline*}
which ends the proof.
\end{itemize}
\end{proof}

\section{\label{sec:Definition-and-properties-solution-d-bar}Definition and
properties of a solution of the \texorpdfstring{$\bar{\partial}$}{\textpartial -bar}-equation}

To solve the $\overline{\partial}$-equation we use the method introduced
in \cite{CDMb}. Let us briefly recall it:

If $f$ is a smooth $\left(0,r\right)$-form $\overline{\partial}$-closed
on $\overline{\Omega}$, then
\[
f(z)=\left(-1\right)^{r+1}\overline{\partial_{z}}\left(\int_{\Omega}f(\zeta)\wedge K_{N}^{r}(z,\zeta)\right)-\int_{\Omega}f(\zeta)\wedge P_{N}^{r}(z,\zeta),
\]
where $K_{N}^{r}$ (resp. $P_{N}^{r}$) is the component of a kernel
$K_{N}$ ($N\geq1$), recalled in formula \eqref{eq:Kernel-KN-general}
below, of bi-degree $\left(0,r-1\right)$ in $z$ and $\left(n,n-r\right)$
in $\zeta$ (resp. the component of $P_{N}$ of be-degree $\left(0,r\right)$
in $z$ and $\left(n,n-r\right)$ in $\zeta$) constructed with the
method of \cite{BA82}, using the Diederich-Fornaess support function
constructed in \cite{Diederich-Fornaess-Support-Func-lineally-cvx}
(see also Theorem 2.2 of \cite{CDMb}) and the function $G(\xi)=\frac{1}{\xi^{N}}$
with a sufficiently large number $N$ (instead of $G(\xi)=\frac{1}{\xi}$
in formula (2.7) of \cite{CDMb}).

Then, the form $\int_{\Omega}f(\zeta)\wedge P_{N}^{r}(z,\zeta)$ is
$\overline{\partial}$-closed and the operator $T$ solving the $\overline{\partial}$-equation
in theorems \ref{thm:Theroem-without-derivative}, \ref{thm:Theorem-Sobolev-Lp-k}
and \ref{thm:Theorem-Sobolev-Lp-k-0-compact-supp} is defined by
\begin{equation}
Tf(z)=\left(-1\right)^{r+1}\int_{\Omega}f(\zeta)\wedge K_{N}^{r}(z,\zeta)-\overline{\partial}^{*}\mathcal{N}\left(\int_{\Omega}f(\zeta)\wedge P_{N}^{r}(z,\zeta)\right),\label{eq:def-op-solv-d-bar}
\end{equation}
where $\overline{\partial}^{*}\mathcal{N}$ is the canonical solution
of the $\overline{\partial}$-equation derived from the theory of
the $\overline{\partial}$-Neumann problem on pseudoconvex domains
of finite type.

By \cite[Lemma 3.2]{ChDu18} and Sobolev lemma we get (see \cite[Lemma 3.3]{ChDu18}):
\begin{sllem}
\label{lem:Sol-d-bar-Neumann-PN}Let $r\geq1$, $-1<\gamma$, $d\in\mathbb{N}$.
Let $f$ be a $\overline{\partial}$-closed $\left(0,r\right)$-form
with coefficients in $L^{1}(\Omega,\delta_{\Omega}^{\gamma})$ and
let $g=\int_{\Omega}f(\zeta)\wedge P_{N}^{r}(z,\zeta)$. Then for
$N\geq N_{0}=N_{0}(d)$ $\overline{\partial}^{*}\mathcal{N}(g)$ is
a solution of the equation $\overline{\partial}u=g$ satisfying $\left\Vert \overline{\partial}^{*}\mathcal{N}(g)\right\Vert _{\mathcal{C}^{d}(\overline{\Omega})}\leq C_{k,N}\left\Vert f\right\Vert _{L^{1}\left(\Omega,\delta^{\gamma}\right)}$.
\end{sllem}

\medskip{}

Finally the proofs of our theorems are reduced to the proofs of estimates
for the operator $T_{K}$ defined by
\begin{equation}
T_{K}(f)=\int_{\Omega}f(\zeta)\wedge K_{N}^{r}(z,\zeta).\label{eq:operator-Tk}
\end{equation}

Let us now recall (from \cite{CDMb}) the explicit formula for the
kernel $K_{N}$:

Let $S_{0}(z,\zeta)$ be the holomorphic support function of Diederich-Fornaess
\cite{Diederich-Fornaess-Support-Func-lineally-cvx} (see also Theorem
2.2 of \cite{CDMb}). Define two $\mathcal{C}^{\infty}$ functions
$\chi_{1}(z,\zeta)=\hat{\chi}\left(\left|z-\zeta\right|\right)$ and
$\chi_{2}(z)=\tilde{\chi}\left(\delta_{\partial\Omega}(z)\right)$
($\delta_{\partial\Omega}$ denoting the distance to the boundary
of $\Omega$) where $\hat{\chi}$ and $\tilde{\chi}$ are $\mathcal{C}^{\infty}$
functions, $0\leq\hat{\chi},\,\tilde{\chi}\leq1$, such that $\hat{\chi}\equiv1$
on $\left[0,\nicefrac{R}{2}\right]$ and $\hat{\chi}\equiv0$ on $\left[R,+\infty\right[$
and $\tilde{\chi}\equiv1$ on $\left[0,\nicefrac{\eta_{1}}{2}\right]$
and $\tilde{\chi}\equiv0$ on $\left[\eta_{0},+\infty\right[$, where
$R$ and $\eta_{1}$ are small enough. Then we define 
\[
\chi(z,\zeta)=\chi_{1}(z,\zeta)\chi_{2}(\zeta)
\]
and
\[
S(z,\zeta)=\chi(z,\zeta)S_{0}(z,\zeta)-\left(1-\chi(z,\zeta)\right)\left|z-\zeta\right|^{2}=\sum_{i=1}^{n}Q_{i}(z,\zeta)\left(z_{i}-\zeta_{i}\right).
\]

Define the two forms $s$ and $Q$ used in \cite{BA82} in the construction
of the Kopelman formula by
\[
s(z,\zeta)=\sum_{i=1}^{n}\left(\overline{\zeta_{i}}-\overline{z_{i}}\right)d\left(\zeta_{i}-z_{i}\right)
\]
and
\[
Q(z,\zeta)=\frac{1}{K_{0}\rho(\zeta)}\sum_{i=1}^{n}Q_{i}(z,\zeta)d\left(\zeta_{i}-z_{i}\right),
\]
where $K_{0}$ is a large constant chosen so that \cite[formula (3.12) p. 204]{ChDu18}
\[
\Re\mathrm{e}\left(\rho(\zeta)+\frac{1}{K_{0}}S(z,\zeta)\right)<\frac{\rho(\zeta)}{2}.
\]

The $Q_{i}(z,\zeta)$ are $\mathcal{C}^{\infty}$ in $\overline{\Omega}\times\overline{\Omega}$
and satisfy the following estimates (\cite[Lemma 3.6]{ChDu18}):
\begin{sllem}
\label{lem:Estimates-Qi-and-derivates}For $z_{0}$ close to $\partial\Omega$,
$\varepsilon$ small and $z,\,\zeta\in P_{\varepsilon}(z_{0})$, in
the coordinate system $\left(\zeta_{i}\right)$ associated to a $\left(z_{0},\varepsilon\right)$-extremal
basis, we have:

\begin{enumerate}
\item $\left|Q_{i}(z,\zeta)\right|+\left|Q_{i}(\zeta,z)\right|\lesssim\frac{\varepsilon}{\tau_{i}}$,
\item $\left|\frac{\partial Q_{i}(z,\zeta)}{\partial\overline{\zeta_{j}}}\right|+\left|\frac{\partial Q_{i}(z,\zeta)}{\partial\zeta_{j}}\right|+\left|\frac{\partial Q_{i}(z,\zeta)}{\partial z_{j}}\right|\lesssim\frac{\varepsilon}{\tau_{i}\tau_{j}}$,
\end{enumerate}
where $\tau_{i}$ are either $\tau_{i}(z,\varepsilon)$, $\tau_{i}(\zeta,\varepsilon)$
or $\tau_{i}\left(z_{0},\varepsilon\right)$.
\end{sllem}

Moreover we have (\cite[Formula (3.12) and Lemma 3.5]{ChDu18})
\begin{equation}
\left|\rho(\zeta)+\frac{1}{K_{0}}S(z,\zeta)\right|\gtrsim\left|\rho(\zeta)\right|+\left|\rho(z)\right|+d(z,\zeta).\label{eq:min_denom_weight_bd+dist}
\end{equation}

Recall that if $\zeta\in\partial\Omega$, $\Re\mathrm{e}S(z,\zeta)<0$
for $z\in\Omega$, and, that $s$ satisfies
\[
\left|z-\zeta\right|^{2}=\left|\left\langle s,z-\zeta\right\rangle \right|\leq C\left|z-\zeta\right|,\,z,\zeta\in\Omega.
\]

Then the kernel $K_{N}$ is (formula (2.7) of \cite{CDMb})
\begin{equation}
K_{N}(z,\zeta)=C_{n}\sum_{k=0}^{n-1}\frac{(n-1)!}{k!}\left(\frac{\rho(\zeta)}{\frac{1}{K_{0}}S(z,\zeta)+\rho(\zeta)}\right)^{N+k}\frac{s(z,\zeta)}{\left|\zeta-z\right|^{2(n-k)}}\wedge\left(dQ\right)^{k}\wedge\left(ds\right)^{n-k-1}.\label{eq:Kernel-KN-general}
\end{equation}

\bigskip{}

The main articulations of the proofs are as follows:

The proof of Theorem \ref{thm:Theroem-without-derivative} is given
in Section \ref{sec:Proof-of-Theorem-2-1} and the proofs of Theorems
\ref{thm:Theorem-Sobolev-Lp-k} and \ref{thm:Theorem-Sobolev-Lp-k-0-compact-supp}
are given in next sections.

To estimate derivatives of $T_{K}(f)$ we note that, for $z\in\Omega$,
the coefficients of $\zeta\rightarrow K_{N}(z,\zeta)$ and their derivatives
up to order $N-1$ are smooth in $\overline{\Omega}\setminus\left\{ z\right\} $
and equal to $0$ when $\zeta\in\partial\Omega$. Thus we define $\widetilde{K}_{N}(z,\zeta)$
the kernel obtained extending the coefficients of $K_{N}$ by $0$
if $(z,\zeta)\in\Omega\times\left(\mathbb{C}^{n}\setminus\overline{\Omega}\right)$.
Then $T_{K}(f)=\int_{\mathbb{C}^{n}}f(\zeta)\wedge\widetilde{K}_{N}^{r}(z,\zeta)$
and, making the change of variable $\zeta-z=\xi$ we can derivate
$T_{K}(f)$ (less than $N-1$ times) without derivating the expression
$\frac{1}{\left|z-\zeta\right|^{2(n-k)}}$. Moreover we use an other
trick (inspired by \cite{wu-convex-98}) on the derivatives of $\frac{1}{K_{0}}S(z,\zeta)+\rho(\zeta)$.

Finally the proofs are done using the following Lemma which is easily
obtained reading the proof of \cite[Lemma 3.10]{ChDu18} (see also
\cite[Appendix B]{RM86}):
\begin{sllem}
\label{lem:kernel-operator-weighted-estimate-Lp-Lq}Let $\Omega$
be a smoothly bounded domain in $\mathbb{C}^{n}$. Let $\mu$ and
$\nu$ be two positive measures on $\Omega$. Let $K$ be a measurable
function on $\Omega\times\Omega$ and let us consider the linear operator
$T$ defined by
\[
Tf(z)=\int_{\Omega}K(z,\zeta)f(\zeta)d\mu(\zeta).
\]

Denote $\delta_{\Omega}$ the distance to the boundary of $\Omega$.
Assume that there exists a real number $s\geq1$ such that:

\begin{enumerate}
\item If there exists a positive number $\kappa_{0}>0$ and, for all $0<\kappa\leq\kappa_{0}$,
a positive constant $C_{\kappa}$ such that,
\[
\int_{\Omega}\left|K(z,\zeta)\right|^{s}\delta_{\Omega}(\zeta)^{-\kappa}d\mu(\zeta)\leq C_{\kappa}\delta_{\Omega}(z)^{-\kappa},\,z\in\Omega
\]
 and
\[
\int_{\Omega}\left|K(z,\zeta)\right|^{s}\delta_{\Omega}(z)^{-\kappa}d\nu(z)\leq C_{\kappa}\delta_{\Omega}(\zeta)^{-\kappa},\,\zeta\in\Omega,
\]
then $T$ is bounded from $L^{p}\left(\Omega,\mu\right)$ to $L^{q}\left(\Omega,\nu\right)$
for all $1<p<\frac{s}{s-1}$ with $\frac{1}{q}=\frac{1}{p}-\frac{s-1}{s}$.
\item If there exists a positive constant $C$ such that $\int_{\Omega}\left|K(z,\zeta)\right|^{s}d\nu(z)\leq C$,
$\zeta\in\Omega$, then $T$ is bounded from $L^{1}\left(\Omega,\mu\right)$
to $L^{q}\left(\Omega,\nu\right)$ with $q=s$.
\end{enumerate}
\medskip{}
\end{sllem}

For $0\leq k\leq n-1$ and $f$ a $\left(0,r\right)$ form ($1\leq r\leq n-1$),
let us denote 
\begin{multline*}
T_{k}f(z)=\frac{(n-1)!}{k!}\int_{\Omega}\left(\frac{\rho(\zeta)}{\frac{1}{K_{0}}S(z,\zeta)+\rho(\zeta)}\right)^{N+k}\\
\frac{s(z,\zeta)}{\left|\zeta-z\right|^{2(n-k)}}\wedge\left(dQ\right)^{k}\wedge\left(ds\right)^{n-k-1}\wedge f(\zeta)=\int_{\Omega}K_{k}(z,\zeta)\wedge f(\zeta)
\end{multline*}
so that $T_{K}f=\sum_{k=0}^{n-1}C_{n}T_{k}f$.

\section{\label{sec:Proof-of-Theorems}Proofs of Theorems}

As sharp estimates of the coefficients of $K_{N}$ using the finite
type hypothesis are made only when $\zeta$ and $z$ are close and
close to $\partial\Omega$, we consider a recovering of $\Omega$
as follows: let $\varepsilon_{0}$ be a small positive number, $P_{l}$,
$1\leq l\leq M$, points in $\partial\Omega$, $V_{l}=B\left(P_{l},\varepsilon_{0}\right)$,
$1\leq l\leq M$, the euclidean balls centered at $P_{l}$ and radius
$\varepsilon_{0}$ and $V_{0}$ an open set relatively compact in
$\Omega$ such that $\left\{ V_{l},\,0\leq l\leq M\right\} $ is an
open covering of $\overline{\Omega}$. Precisely $\varepsilon_{0}$
is chosen sufficiently small so that:
\begin{itemize}
\item For $z$ and $\zeta$ in the balls $2V_{l}=B\left(P_{l},2\varepsilon_{0}\right)$,
$1\leq l\leq M$, all the properties and estimates associated to the
finite type hypothesis are valid (see \eqref{eq:definition-tau-general}
to \eqref{eq:3.16} and Lemma \ref{lem:Estimates-Qi-and-derivates})
including Lemmas \ref{lem:estimates-derivatives-Qi-in-kernel} to
\ref{lem:Normal-derivative-weight-Wu};
\item For $z$ and $\zeta$ in $2V_{l}\cap\Omega$, $1\leq l\leq M$, $Q(z,\zeta)$
is holomorphic in the $z$ variable.
\end{itemize}
Then we choose smooth cut-off functions $\chi_{l}\in\mathcal{C}_{0}^{\infty}\left(V_{l}\right)$,
$\sum_{l=0}^{M}\chi_{l}\equiv1$ on $\overline{\Omega}$ and we write
\[
T_{k}f(z)=\sum_{l=0}^{M}\int_{\Omega}K_{k}(z,\zeta)\chi_{l}(\zeta)\wedge f(\zeta)=\sum_{l=0}^{M}T_{k}^{l}f(z)
\]
so
\[
T_{K}f(z)=\sum_{l=0}^{M}\sum_{k=0}^{n-1}C_{n}T_{k}^{l}f(z).
\]

By standard regularization procedure, to prove Theorem \ref{thm:Theroem-without-derivative},
Theorem \ref{thm:Theorem-Sobolev-Lp-k} and Theorem \ref{thm:Theorem-Sobolev-Lp-k-0-compact-supp}
under condition (2) or (3) (resp. Theorem \ref{thm:Theorem-Sobolev-Lp-k-0-compact-supp}
under condition (1)) we can assume that the coefficients of $f$ are
$\mathcal{C}^{\infty}$-smooth in $\overline{\Omega}$ (resp. in $\mathcal{D}(\Omega)$).
So in all the proofs below the coefficients of the form $f$ are in
$\mathcal{C}^{\infty}\left(\overline{\Omega}\right)$.

Moreover, for $l\geq1$, if $z\in2V_{l}$, using the holomorphic property
of $Q$, an elementary calculus shows that, in the above expression
of $T_{K}f(z)$, $T_{k}^{l}f(z)$ appears only for $k\leq n-r$ and
the kernel $K_{k}(z,\zeta)\chi_{l}(\zeta)$ is equal to
\begin{multline}
K_{k}(z,\zeta)\chi_{l}(\zeta)=\chi_{l}(\zeta)C_{n,k}\left(\frac{\rho(\zeta)}{\frac{1}{K_{0}}S(z,\zeta)+\rho(\zeta)}\right)^{N+k}\\
\frac{s(z,\zeta)}{\left|\zeta-z\right|^{2(n-k)}}\wedge\left(\overline{\partial}_{\zeta}Q(z,\zeta)\right)^{k}\wedge\left(\overline{\partial}_{\zeta}s(z,\zeta)\right)^{n-k-r}\wedge\left(\overline{\partial}_{z}s(z,\zeta)\right)^{r-1}\label{eq:K-l-l-z-zeta-closes}
\end{multline}
with $s(z,\zeta)=\sum_{i}\left(\overline{\zeta}_{i}-\overline{z}_{i}\right)d\zeta_{i}$
and $Q(z,\zeta)=\frac{1}{K_{0}\rho(\zeta)}\sum_{i}Q_{i}(z,\zeta)d\zeta_{i}$.
\begin{sllem}
\label{lem:estimates-derivatives-Qi-in-kernel}Let $D$ be a derivative
of order $j$ in $z$ and $\zeta$. Denote
\[
F(z,\zeta)=\rho(\zeta)\partial_{\overline{\zeta}_{a_{s}}}Q_{b_{s}}(z,\zeta)+Q_{b_{s}}(z,\zeta)\partial_{\overline{\zeta}_{a_{s}}}\rho(\zeta).
\]

Then for $z$ and $\zeta$ close and close to the boundary of $\Omega$
\begin{enumerate}
\item $\left|F(z,\zeta)\right|\apprle\frac{\alpha^{2}}{\tau_{a_{s}}\tau_{b_{s}}}$;
\item $\left|DF(z,\zeta)\right|\apprle\frac{1}{\left|\rho(\zeta)\right|}\frac{\alpha^{2}}{\tau_{a_{s}}\tau_{b_{s}}}$
if $j=1$;
\item $\left|DF(z,\zeta)\right|\apprle1\apprle\frac{1}{\left|\rho(\zeta)\right|^{2}}\frac{\alpha^{2}}{\tau_{a_{s}}\tau_{b_{s}}}$
for $j\geq2$;
\end{enumerate}
where $\alpha=\delta_{\Omega}(z)+d(z,\zeta)\simeq\delta_{\Omega}(\zeta)+d(z,\zeta)$
and $\tau_{i}=\tau_{i}(z,\alpha)\simeq\tau_{i}(\zeta,\alpha)$.
\end{sllem}

\begin{proof}
Comes from (3) of Lemma \ref{lem:3.4-maj-deriv-rho-equiv-tho-i-z-zeta}
and Lemma \ref{lem:Estimates-Qi-and-derivates}.
\end{proof}

\subsection{\label{sec:Proof-of-Theorem-2-1}Proof of Theorem \ref{thm:Theroem-without-derivative}}

The result stated here is slightly stronger than the one given in
\cite{ChDu18} and we give a detailed proof. We assume $N>2\left|\gamma\right|+1+N_{0}$.
\begin{sllem}
\label{lem:pointwise-estimate-kernel-Kk}If $z$ and $\zeta$ close
to the boundary and close together
\[
\left|K_{k}(z,\zeta)\right|\apprle\left|\frac{\rho(\zeta)}{\alpha}\right|^{N-1}\frac{1}{\left|z-\zeta\right|^{2(n-k)-1}}\prod_{i=1}^{k}\tau_{i}^{-2}
\]
where $\tau_{i}=\tau_{i}(z,\alpha)$ or $\tau_{i}=\tau_{i}(\zeta,\alpha)$
with $\alpha=\delta(z)+d(z,\zeta)\simeq\delta(\zeta)+d(z,\zeta)$,
otherwise
\[
\left|K_{k}(z,\zeta)\right|\apprle\frac{1}{\left|z-\zeta\right|^{2(n-k)-1}}.
\]
\end{sllem}

This Lemma follows easily formula \eqref{eq:K-l-l-z-zeta-closes}
and Lemma \ref{lem:estimates-derivatives-Qi-in-kernel}.
\begin{proof}[Proof of (1) of Theorem \ref{thm:Theroem-without-derivative}]
Let us consider the operators $T_{k}^{l}$, $0\leq l\leq M$.

We use Lemma \ref{lem:kernel-operator-weighted-estimate-Lp-Lq}, with
$s=1+\mu$ and $0\leq\mu<1/(2n-1)$, writing
\[
T_{k}^{l}f(z)=\int_{\Omega}\widetilde{K}_{k}^{l}(z,\zeta)\left(f(\zeta)\delta(\zeta)^{\frac{\gamma-\gamma'}{p}}\right)d\mu(\zeta)
\]
with $\widetilde{K}_{k}^{l}(z,\zeta)=\chi_{l}(\zeta)K_{k}(z,\zeta)\delta(\zeta)^{\frac{-\gamma+\gamma'}{p}-\gamma'}$
and $d\mu(\zeta)=\delta(\zeta)^{\gamma'}d\lambda(\zeta)$:

Thus to prove that $T_{k}$ maps $L^{p}\left(\delta^{\gamma}d\lambda\right)$
into $L^{q}\left(\delta^{\gamma'}d\lambda\right)$ (that is $\widetilde{T}_{k}^{l}$
defined by $\widetilde{T}_{k}^{l}g(z)=\int_{\Omega}\widetilde{K}_{k}^{l}(z,\zeta)g(\zeta)d\mu(\zeta)$
maps $L^{p}(d\mu)$ into $L^{q}(d\nu)$ with $d\nu(z)=\delta(z)^{\gamma'}d\lambda(z)$)
we have to estimate the two following integrals:
\[
I_{1}(z)=\int_{\Omega}\left|\widetilde{K}_{k}^{l}(z,\zeta)\right|^{1+\mu}\delta(\zeta)^{\gamma'-\varepsilon}d\lambda(\zeta),\,0\leq l\leq M
\]
and
\[
I_{2}(\zeta)=\int_{\Omega}\left|\widetilde{K}_{k}^{l}(z,\zeta)\right|^{1+\mu}\delta(z)^{\gamma'-\varepsilon}d\lambda(z),\,0\leq l\leq M,
\]
for $0\leq\varepsilon<\min\left\{ \gamma'+1,1\right\} $.

If $l=0$ the estimate of $I_{2}(z)$ is trivial because
\begin{itemize}
\item if $\left|z-\zeta\right|\leq\mathrm{dist}\left(V_{0},\partial\Omega\right)/2$,
$\delta_{\Omega}(z)\simeq\delta_{\Omega}(\zeta)$ and $\delta_{\Omega}(\zeta)^{-(1+\mu)\left(\frac{-\gamma+\gamma'}{p}-\gamma'\right)}\delta_{\Omega}(z)^{\gamma'-\varepsilon}$
is bounded,
\item if $\left|z-\zeta\right|>\mathrm{dist}\left(V_{0},\partial\Omega\right)/2$,
$\left|K_{k}(z,\zeta)\right|^{1+\mu}\delta_{\Omega}(\zeta)^{-(1+\mu)\left(\frac{-\gamma+\gamma'}{p}-\gamma'\right)}$
is bounded and $\gamma'-\varepsilon>-1$,
\end{itemize}
so $I_{2}(\zeta)\apprle_{\varepsilon,\mu}1\apprle\delta_{\Omega}(\zeta)^{-\varepsilon}$.

If $l\geq1$ and $z\notin2V_{l}$then:

$\left|K_{k}(z,\zeta)\right|^{1+\mu}\delta_{\Omega}(\zeta)^{-(1+\mu)\left(\frac{-\gamma+\gamma'}{p}-\gamma'\right)}$
is bounded, trivially if $\frac{-\gamma+\gamma'}{p}-\gamma'\leq0$
and otherwise by Lemma \ref{lem:pointwise-estimate-kernel-Kk}, and
we obtain the same result because $\gamma'-\varepsilon>-1$.

Similar result is easily obtained for $I_{1}(z)$.

So we only have to estimate these integrals when $l\geq1$ and $z\in2V_{l}$.

For $-1<\gamma'\leq\gamma$, $\gamma-\gamma'<p/m$, $\gamma<m+n+1$
let
\[
\mu_{0}(k)=\left\{ \begin{array}{cc}
\frac{1-\frac{m}{p}\left(\gamma-\gamma'\right)}{2\left(n-k\right)-1+m(k+1+\gamma'+\frac{\gamma-\gamma'}{p})} & \mathrm{if}\,k\neq0\\
\frac{1-\frac{\text{\ensuremath{\gamma}-\ensuremath{\gamma}'}}{p}}{2n-1+\gamma'+\frac{\gamma-\gamma'}{p}} & \mathrm{if}\,k=0
\end{array}\right.,
\]
\begin{spprop}
For $l\geq1$, $z\in2V_{l}$ and $\mu\leq\mu_{0}(k)$, $I_{1}(z)\apprle_{\varepsilon}\delta(z)^{-\varepsilon}$
if $0<\varepsilon<\gamma'+1$.
\end{spprop}

\begin{proof}
Denote $P_{o}(z)=P(z,\delta(z))$ and $P_{i}(z)=P\left(z,2^{i}\delta(z)\right)\setminus P\left(z,2^{i-1}\delta(z)\right)$
for $i\geq1$. Then, using Lemma \ref{lem:pointwise-estimate-kernel-Kk}
or a direct calculus ($N>2\left|\gamma\right|$), if $\zeta\in P_{i}(z)$,
by \eqref{eq:dist-boundary-in-polydisc}
\[
\left|\widetilde{K}_{k}^{l}(z,\zeta)\right|^{1+\mu_{0}}\delta(\zeta)^{\gamma'-\varepsilon}\apprle\left|K_{k}(z,\zeta)\right|^{1+\mu_{0}}\text{\ensuremath{\left(2^{i}\delta(z)\right)}}^{\left(1+\mu_{0}\right)\frac{\gamma'-\gamma}{p}-\mu_{0}\gamma'-\varepsilon}.
\]
 \begin{enumerate}\item Case $k\neq0$\\
By Lemma \ref{lem:Lemma-5-1}, \eqref{eq:integral-mod(z-zeta)-inverse-power-1+mu},
properties \eqref{eq:estimation-tau-i-1-n}
\begin{eqnarray*}
\int_{P_{i}(z)}\left|K_{k}(z,\zeta)\right|^{1+\mu_{0}}d\lambda(\zeta) & \apprle & \tau_{k+1}\left(z,2^{i}\delta(z)\right)^{1-\mu_{0}(2(n-k)-1)}\prod_{j=1}^{k}\left(\tau_{j}^{-2}\left(z,2^{i}\delta(z)\right)\right)^{\mu_{0}}\\
 & \apprle & \left(2^{i}\delta(z)\right)^{\frac{1}{m}\left(1-\mu_{0}(2(n-k)-1)\right)-\mu_{0}(k+1)}.
\end{eqnarray*}
Then
\begin{multline*}
\int_{P_{i}(z)}\left|\widetilde{K}_{k}^{l}(z,\zeta)\right|^{1+\mu_{0}}\delta(\zeta)^{\gamma'-\varepsilon}d\lambda(\zeta)\\
\apprle\left(2^{i}\delta(z)\right)^{\frac{1}{m}\left(1-\mu_{0}(2(n-k)-1)\right)-\mu_{0}(k+1)+(1+\mu_{0})\frac{\gamma'-\gamma}{p}-\mu_{0}\gamma'-\varepsilon}\\
\apprle\left(2^{i}\delta(z)\right)^{-\varepsilon-\mu_{0}\left(\frac{2(n-k)-1}{m}+k+1+\frac{\gamma-\gamma'}{p}+\gamma'\right)+\frac{1}{m}+\frac{\gamma'-\gamma}{p}}\\
\apprle\left(2^{i}\delta(z)\right)^{-\varepsilon}.
\end{multline*}
\\
\item Case $k=0$\\
Similarly
\begin{eqnarray*}
\int_{P_{i}(z)}\left|\widetilde{K}_{0}^{l}(z,\zeta)\right|^{1+\mu_{0}}\delta(\zeta)^{\gamma'-\varepsilon}d\lambda(\zeta) & \apprle & \left(2^{i}\delta(z)\right)^{-\varepsilon-\mu_{0}\gamma'-(2n-1)\mu_{0}+\left(1+\mu_{0}\right)\frac{\gamma'-\gamma}{p}}\\
 & \apprle & \left(2^{i}\delta(z)\right)^{-\varepsilon}.
\end{eqnarray*}
\\
\end{enumerate}Thus $I_{1}(z)\apprle_{\varepsilon}\delta(z)^{-\varepsilon}$
which ends the proof of the Proposition.\vspace{1em}
\end{proof}
\begin{spprop}
For $l\geq1$, $z\in2V_{l}$ $I'_{2}(\zeta)=\int_{\Omega\cap2V_{l}}\left|\widetilde{K}_{k}^{l}(z,\zeta)\right|^{1+\mu}\delta(z)^{\gamma'-\varepsilon}d\lambda(z)\apprle_{\varepsilon}\delta(\zeta)^{-\varepsilon}$
with $0<\varepsilon<\min\left\{ \gamma'+1,1\right\} $ and $\mu\leq\mu_{0}(k)$
or $\varepsilon=0$ with $\mu\leq\mu_{0}(k)$ if $\gamma<0$ and $\mu<\mu_{0}(k)$
otherwise.
\end{spprop}

\begin{proof}
Assume $k\neq0$. By Lemma \ref{lem:pointwise-estimate-kernel-Kk}
or trivially if $\frac{-\gamma+\gamma'}{p}-\gamma'\leq0$, for $z\in P_{i}(\zeta)$
\[
\left|\widetilde{K}_{k}^{l}(z,\zeta)\right|^{1+\mu}\delta(z)^{\gamma'-\varepsilon}\apprle\left|K_{k}(z,\zeta)\right|^{1+\mu}\left(2^{i}\delta_{\Omega}(\zeta)\right)^{\frac{-\gamma+\gamma'}{p}-\gamma'}\delta(z)^{\gamma'-\varepsilon}
\]
and, by \eqref{eq:3.16}, $\gamma'-\varepsilon>-1$ implies
\[
I'_{2}(\zeta)\apprle\left(2^{i}\delta_{\Omega}(\zeta)\right)^{-\varepsilon-\mu\left(\frac{2(n-k)-1}{m}+k+1+\frac{\gamma-\gamma'}{p}+\gamma'\right)+\frac{1}{m}+\frac{\gamma'-\gamma}{p}}.
\]

The case $k=0$ is easily treated as in the previous proposition which
ends the proof.
\end{proof}
The proof of (1) of Theorem \ref{thm:Theroem-without-derivative}
is complete.
\end{proof}
\begin{proof}[Proof of (2) of Theorem \ref{thm:Theroem-without-derivative}]
We give two different methods for the study of $T_{0}f$ and $T_{k}f$,
$k\geq1$.
\begin{itemize}
\item \emph{Case $k\geq1$. }As in \cite{ChDu18} we use the Hardy-Littlewood
lemma, so, we have to prove the following inequality 
\[
\left|\nabla_{z}T_{k}f\right|\apprle\delta_{\Omega}(z)^{\alpha-1},
\]
if $p>m(\gamma+n)+2-(m-2)(r-1)$ with $\alpha=\frac{1}{m}\left[1-\frac{m(\gamma+n)+2-(m-2)(r-1)}{p}\right]$.\\
By Hölder inequality it follows from
\begin{equation}
\int_{\Omega}\left|\nabla_{z}K_{k}(z,\zeta)\chi_{l}(\zeta)\right|^{\frac{p}{p-1}}\delta_{\Omega}(\zeta)^{-\frac{\gamma}{p-1}}\apprle\delta_{\Omega}(z)^{\frac{p(\alpha-1)}{p-1}}.\label{eq:Lipschtz-est-kernel}
\end{equation}
\\
If $l=0$
\begin{eqnarray*}
\left|\nabla_{z}K_{k}(z,\zeta)\chi_{l}(\zeta)\right|^{\frac{p}{p-1}}\delta_{\Omega}(\zeta)^{-\frac{\gamma}{p-1}} & \apprle & \frac{1}{\left|z-\zeta\right|^{(2n-1)\frac{p}{p-1}}}\\
 & \apprle & \frac{1}{\left|z-\zeta\right|^{2n\frac{2n-2}{2n-1}}}
\end{eqnarray*}
and \eqref{eq:Lipschtz-est-kernel} is trivial.\\
If $l\geq1$ and $z\notin2V_{l}$ $\left|\nabla_{z}K_{k}(z,\zeta)\chi_{l}(\zeta)\right|$
is bounded (by Lemma \ref{lem:pointwise-estimate-kernel-Kk}) and
\eqref{eq:Lipschtz-est-kernel} is clear because $\frac{\gamma}{p-1}<1$.\\
In the last case, the $Q_{i}$ being holomorphic in $z$, it remains
only the terms corresponding to $k\leq n-r$ and by Lemma \ref{lem:pointwise-estimate-kernel-Kk},
Lemma \ref{lem:estimates-derivatives-Qi-in-kernel}, the fact that
for $s\neq s'$ $a_{s}\neq a_{s'}$ and $b_{s}\neq b_{s'}$ and the
increasing property of the $\tau_{i}$, if $z\in2V_{l}$ we have
\begin{multline*}
\left|\nabla_{z}K_{k}(z,\zeta)\chi_{l}(\zeta)\delta_{\Omega}(\zeta)^{-\gamma/p}\right|\apprle\\
\frac{1}{\left|z-\zeta\right|^{2(n-k)-1}}\left(\frac{1}{\delta(z)+d(z,\zeta)}\right)^{1+\frac{\gamma}{p}}\prod_{i=1}^{k}\tau_{i}(z,\delta(z)+d(z,\zeta))^{-2}\\
+\frac{1}{\left|z-\zeta\right|^{2(n-k)}}\left(\frac{1}{\delta(z)+d(z,\zeta)}\right)^{\frac{\gamma}{p}}\prod_{i=1}^{k}\tau_{i}(z,\delta(z)+d(z,\zeta))^{-2}\\
=W_{1}+W_{2}.
\end{multline*}
\\
Let $P_{0}(z)=P(z,\delta(z))$ and $P_{j}(z)=P\left(z,2^{j}\delta(z)\right)\setminus P\left(z,2^{j-1}\delta(z)\right)$.
Then, if $\zeta\in P_{j}(z)$, by \eqref{eq:dist-boundary-in-polydisc},
\[
\left|W_{1}\right|^{\frac{p}{p-1}}\apprle\frac{1}{\left(\left|z-\zeta\right|^{2(n-k)-1}\right)^{\frac{p}{p-1}}}\left(2^{j}\delta(z)\right)^{-\frac{p}{p-1}\left(1+\frac{\gamma}{p}\right)}\prod_{i=1}^{k}\tau_{i}(z,\delta(z)+d(z,\zeta))^{-2\frac{p}{p-1}}
\]
and by Lemma \ref{lem:Lemma-5-1}
\[
\int_{P_{j}(z)}\frac{1}{\left(\left|z-\zeta\right|^{2(n-k)-1}\right)^{\frac{p}{p-1}}}d\lambda(\zeta)\apprle\tau_{k+1}(z,2^{i}\delta_{\Omega}(z))^{1-(2(n-k)-1)\frac{1}{p-1}}\prod_{i=1}^{k}\tau_{i}(z,2^{j}\delta_{\Omega}(z))^{2}.
\]
\\
Then
\begin{eqnarray*}
\int_{P_{j}(z)}\left|W_{1}\right|^{\frac{p}{p-1}} & \apprle & \tau_{k+1}(z,2^{j}\delta_{\Omega}(z))^{1-\frac{(2(n-k)-1)}{p-1}}\left(2^{j}\delta(z)\right)^{1-\frac{(2(n-k)-1)}{p-1}}\prod_{i=1}^{k}\tau_{i}(z,2^{j}\delta_{\Omega}(z))^{\frac{-2}{p-1}}\\
 & \apprle & \left(2^{j}\delta(z)\right)^{\frac{1}{m}\left(1-\frac{(2(n-k)-1)}{p-1}\right)-\frac{p+\gamma+k+1}{p-1}}\,\mathrm{by\,\eqref{eq:estimation-tau-i-1-n}}\\
 & \apprle & \left(2^{j}\delta(z)\right)^{\frac{p}{m(p-1)}\left[1-m-\frac{m(\gamma+n-r+1)+2r}{p}\right]}\mathrm{\,(because}\,k\leq n-r\mathrm{)}\\
 & = & \left(2^{j}\delta(z)\right)^{\frac{p(\alpha-1)}{p-1}}.
\end{eqnarray*}
\\
Similarly, if $\zeta\in P_{j}(z)$,
\[
\left|W_{2}\right|^{\frac{p}{p-1}}\apprle\frac{1}{\left(\left|z-\zeta\right|^{2(n-k)}\right)^{\frac{p}{p-1}}}\left(2^{j}\delta(z)\right)^{\frac{-\gamma}{p-1}}\prod_{i=1}^{k}\tau_{i}\left(z,2^{j}\delta(z)\right)^{-\frac{2p}{p-1}}
\]
and as Lemma \ref{lem:Lemma-5-1} implies
\[
\int_{P_{j}(z)}\frac{1}{\left|z-\zeta\right|^{\frac{2p(n-k)}{p-1}}}\apprle\tau_{k}\left(z,2^{j}\delta(z)\right)^{2-\frac{2(n-k)}{p-1}}\prod_{i=1}^{k-1}\tau_{i}\left(z,2^{j}\delta(z)\right)^{2}
\]
we get
\begin{eqnarray*}
\int_{P_{j}(z)}\left|W_{2}\right|^{\frac{p}{p-1}} & \apprle & \left\{ \begin{array}{cc}
\left(2^{j}\delta(z)\right)^{-\frac{1}{p-1}\left(2n-2+\gamma\right)} & \mathrm{if\,}k>1\\
\left(2^{j}\delta(z)\right)^{-\frac{1}{p-1}(2n+\gamma)} & \mathrm{if\,}k=1
\end{array}\right.\\
 & \apprle & \left(2^{j}\delta(z)\right)^{\frac{p\left(\alpha'-1\right)}{p-1}}
\end{eqnarray*}
with $\alpha'=1-\frac{2n+\gamma}{p}>\alpha$.\vspace{1em}
\item \emph{Case $k=0$.} In that case as we cannot derive $T_{0}f$ we
make a direct computation for kernels of the form
\[
K_{0}(z,\zeta)=\frac{\zeta_{t}-z_{t}}{\left|z-\zeta\right|^{2n}}\left(\frac{\rho(\zeta)}{\frac{1}{K_{0}}S(z,\zeta)+\rho(\zeta)}\right)^{N}\chi_{l}(\zeta).
\]
\\
We have to estimate the difference
\begin{multline*}
\left|u_{0}(z_{0})-u_{0}(z_{1})\right|\leq\int_{\Omega}\left|K_{0}(z_{0},\zeta)-K_{0}(z_{1},\zeta)\right|\left|f(\zeta)\right|d\lambda(\zeta)\\
\apprle\left(\int_{\Omega}\left|f(\zeta)\right|^{p}\delta_{\Omega}(\zeta)^{\gamma}d\lambda(\zeta)\right)^{1/p}\\
\left(\int_{\Omega}\left|K_{0}(z_{0},\zeta)-K_{0}(z_{1},\zeta)\right|^{\frac{p}{p-1}}\delta_{\Omega}(\zeta)^{-\frac{\gamma}{p-1}}d\lambda(\zeta)\right)^{\frac{p-1}{p}}.
\end{multline*}
\\
Let us write, $z_{0}$ and $z_{1}$ being fixed,
\begin{multline*}
K_{0}(z_{0},\zeta)-K_{0}(z_{1},\zeta)=\\
\left(\frac{\zeta_{t}-z_{0,t}}{\left|z_{0}-\zeta\right|^{2n}}-\frac{\zeta_{t}-z_{1,t}}{\left|z_{1}-\zeta\right|^{2n}}\right)\left(\frac{\rho(\zeta)}{\frac{1}{K_{0}}S(z_{0},\zeta)+\rho(\zeta)}\right)^{N}\chi_{l}(\zeta)\\
+\left(\left(\frac{\rho(\zeta)}{\frac{1}{K_{0}}S(z_{0},\zeta)+\rho(\zeta)}\right)^{N}-\left(\frac{\rho(\zeta)}{\frac{1}{K_{0}}S(z_{1},\zeta)+\rho(\zeta)}\right)^{N}\right)\frac{\zeta_{t}-z_{1,t}}{\left|z_{1}-\zeta\right|^{2n}}\chi_{l}(\zeta)\\
=K_{0}^{1}(\zeta)+K_{0}^{2}(\zeta).
\end{multline*}
\\
Let us denote $T_{0}$ the real tangent space to $\rho$ and $\eta_{0}$
the inward real normal to $\rho$ at the point $z_{0}$. Let $Z\in T_{0}$
and $W\in\partial\Omega$ such that $\zeta=W+t\eta_{0}$ and $W=Z+t_{1}\eta_{0}$.\\
Let us first estimate $I\left(z_{0},z_{1}\right)=\int_{\Omega}\left|K_{0}^{1}(\zeta)\right|^{\frac{p}{p-1}}\delta_{\Omega}(\zeta)^{-\frac{\gamma}{p-1}}d\lambda(\zeta)$.\\
Let $P_{0}\left(z_{0}\right)=B\left(z_{0},3\left|z_{0}-z_{1}\right|\right)$,
$P_{j}\left(z_{0}\right)=B\left(z_{0},3^{j+1}\left|z_{0}-z_{1}\right|\right)\setminus B\left(z_{0},3^{j}\left|z_{0}-z_{1}\right|\right)$
and $Q_{j}\left(z_{0}\right)=P_{j}\left(z_{0}\right)\cap T_{0}$ where
$B(z,r)$ denotes the euclidean ball of center $z$ and radius $r$
in $\mathbb{C}^{n}$.\\
Let us first consider $I_{0}\left(z_{0},z_{1}\right)=\int_{P_{0}\left(z_{0}\right)}\left|K_{0}^{1}(\zeta)\right|^{\frac{p}{p-1}}\delta_{\Omega}(\zeta)^{-\frac{\gamma}{p-1}}d\lambda(\zeta)$:
\begin{multline*}
I_{0}\left(z_{0},z_{1}\right)\apprle\int_{P_{0}\left(z_{0}\right)}\frac{\delta_{\Omega}(\zeta)^{-\frac{\gamma}{p-1}}}{\left|z_{0}-\zeta\right|^{(2n-1)\frac{p}{p-1}}}d\lambda(\zeta)\\
+\int_{P_{0}\left(z_{0}\right)}\frac{\delta_{\Omega}(\zeta)^{-\frac{\gamma}{p-1}}}{\left|z_{1}-\zeta\right|^{(2n-1)\frac{p}{p-1}}}d\lambda(\zeta)=I_{0}^{1}\left(z_{0},z_{1}\right)+I_{0}^{2}\left(z_{0},z_{1}\right).
\end{multline*}
To calculate the integral $I_{0}^{1}\left(z_{0},z_{1}\right)$ we
use the coordinate system $\left(Z,t\right)\in T\times\mathbb{R}\eta_{0}$:
$\delta(\zeta)\simeq t$, $\left|z_{0}-\zeta\right|\simeq\left|Z-z_{0}\right|+\left|t+t_{1}\right|$
(note that $\frac{\gamma}{p-1}<1$)
\begin{eqnarray*}
I_{0}^{1}\left(z_{0},z_{1}\right) & \apprle & \int_{\left|Z-z_{0}\right|\leq3\left|z_{0}-z_{1}\right|}\left|Z-z_{0}\right|^{-\frac{1}{p-1}\left(\gamma+(2n-1)p\right)+1}\\
 & \apprle & \left|z_{0}-z_{1}\right|^{\frac{1}{p-1}\left(p-(\gamma+2n)\right)}.
\end{eqnarray*}
\\
Now it is easy to see that the same estimate is valid for $I_{0}^{2}\left(z_{0},z_{1}\right)$.\\
Let us now look at $I_{j}\left(z_{0},z_{1}\right)=\int_{P_{j}\left(z_{0}\right)}\left|K_{0}^{1}(\zeta)\right|^{\frac{p}{p-1}}\delta_{\Omega}(\zeta)^{-\frac{\gamma}{p-1}}d\lambda(\zeta)$,
$j\geq1$: using that
\[
\left|\frac{\zeta_{t}-z_{0,t}}{\left|z_{0}-\zeta\right|^{2n-1}}-\frac{\zeta_{t}-z_{1,t}}{\left|z_{1}-\zeta\right|^{2n-1}}\right|\apprle\left|z_{0}-z_{1}\right|\left(3^{j}\left|z_{0}-z_{1}\right|\right)^{-2n}
\]
we get
\begin{eqnarray*}
I_{j}\left(z_{0},z_{1}\right) & \apprle & \left|z_{0}-z_{1}\right|^{\frac{p}{p-1}}\left(3^{j}\left|z_{0}-z_{1}\right|\right)^{-\frac{2np}{p-1}+2n-1-\frac{\gamma}{p-1}+1}\\
 & \apprle & 3^{\frac{-j(2n+\gamma)}{p-1}}\left|z_{0}-z_{1}\right|^{\frac{p-(2n+\gamma)}{p-1}}
\end{eqnarray*}
\\
Thus $\sum_{j}I_{j}\left(z_{0},z_{1}\right)\apprle\left|z_{0}-z_{1}\right|^{\frac{p-(2n+\gamma)}{p-1}}$.
Finally we have
\[
\left(\int_{\Omega}\left|K_{0}^{1}(\zeta)\right|^{\frac{p}{p-1}}\delta_{\Omega}(\zeta)^{-\frac{\gamma}{p-1}}d\lambda(\zeta)\right)^{\frac{p-1}{p}}\apprle\left|z_{0}-z_{1}\right|^{1-\frac{2n+\gamma}{p}}.
\]
\\
To estimate 
\[
\int_{\Omega}\left|K_{0}^{2}(\zeta)\right|^{\frac{p}{p-1}}\delta_{\Omega}(\zeta)^{-\frac{\gamma}{p-1}}d\lambda(\zeta)
\]
 we use the pseudo-balls $P_{0}\left(z_{1}\right)=P\left(z_{1},\left|z_{0}-z_{1}\right|\right)$
and $P_{k}\left(z_{1}\right)=P\left(z_{1},2^{k}\left|z_{0}-z_{1}\right|\right)\setminus P\left(z_{1},2^{k-1}\left|z_{0}-z_{1}\right|\right)$.\\
For $\int_{P_{0}\left(z_{1}\right)}\left|K_{0}^{2}(\zeta)\right|^{\frac{p}{p-1}}\delta_{\Omega}(\zeta)^{-\frac{\gamma}{p-1}}d\lambda(\zeta)$
we argue as for $I_{0}^{1}\left(z_{0},z_{1}\right)$ using Lemma \ref{lem:Lemma-5-1}\@.\\
This ends the proof because $\alpha<1-\frac{2n+\gamma}{p}$.
\end{itemize}
\end{proof}

\subsection{Proofs of Theorems \ref{thm:Theorem-Sobolev-Lp-k} and \ref{thm:Theorem-Sobolev-Lp-k-0-compact-supp}}

Denotes $f_{l}=\chi_{l}f$, so
\[
T_{K}f(z)=\sum_{l=0}^{M}\sum_{k=0}^{n-1}\int_{\Omega}K_{k}(z,\zeta)\wedge f_{l}(\zeta).
\]

As explained in Section \ref{sec:Definition-and-properties-solution-d-bar}
we extend the kernels $K_{k}$ to $\left(z,\zeta\right)\in\Omega\times\mathbb{C}^{n}$
with
\[
K_{k}(z,\zeta)=\left\{ \begin{array}{cc}
K_{k}(z,\zeta) & \text{if }(z,\zeta)\in\Omega\times\Omega,\\
0 & \text{if }(z,\zeta)\in\Omega\times\left(\mathbb{C}^{n}\setminus\Omega\right),
\end{array}\right.
\]
and we also extend the coefficients of the $\left(0,r\right)$-form
$f$ in $\mathcal{C}_{0}^{\infty}\left(\mathbb{C}^{n}\right)$. Thus,
for $0\leq k\leq n-r$, making the change of variable $\xi=z-\zeta$
we get
\begin{equation}
T_{k}^{l}f(z)=-\int_{\mathbb{C}^{n}}K_{k}(z,z-\xi)\wedge f_{l}(z-\xi).\label{eq:Tkf-integrate-Cn}
\end{equation}

Note that, $f$ being assumed smooth, the estimate \eqref{eq:min_denom_weight_bd+dist}
implies that $T_{k}^{l}f(z)$ is smooth in $\Omega$ and it's derivatives
are given derivating under the sign of integration and does not depend
on the order of derivation.

As $\overline{\partial}\left(\sum_{k=0}^{n-1}T_{k}f\right)=f-P_{N}f$,
by Lemma \ref{lem:Sol-d-bar-Neumann-PN}, all derivatives of $\sum_{k=0}^{n-r}T_{k}f$
involving an anti-holomorphic derivative are derivatives of a coefficient
of $f$ up to a function smooth in $\overline{\Omega}$, so we need
to estimate derivatives of the $T_{k}f$ involving only holomorphic
derivatives (This is not necessarily but useful to simplify the notation
in the proofs that follow).

To fix notations, for any integer $J$ let $D_{J}=\BCo_{j=1}^{J}\frac{\partial}{\partial z_{i_{j}}}=\frac{\partial^{J}}{\partial z_{i_{1}}\ldots\partial z_{i_{J}}}$,
$i_{j}\in\left\{ 1,\ldots,n\right\} $, a derivative of order $J$
in the $z$ variable and denote $\Delta_{J}=\BCo_{j=1}^{J}\left(\frac{\partial}{\partial z_{i_{j}}}+\frac{\partial}{\partial\zeta_{i_{j}}}\right)$
the derivative obtained replacing $\frac{\partial}{\partial z_{i_{j}}}$
by $\frac{\partial}{\partial z_{i_{j}}}+\frac{\partial}{\partial\zeta_{i_{j}}}$,
$1\leq j\leq J$, in $D_{J}$. Then calculating $D_{J}T_{k}^{l}f$
with \eqref{eq:Tkf-integrate-Cn} and making the change of variable
$z-\xi=\zeta$ we obtain that $D_{J}T_{k}^{l}f$ is a sum of integrals
of the form
\begin{multline}
I_{J,k}^{l}=\int_{\Omega}\frac{z_{t}-\zeta_{t}}{\left|z-\zeta\right|^{2(n-k)}}\\
\Delta_{J}\left(\rho(\zeta)^{N-k}\left(\frac{1}{\frac{1}{K_{0}}S(z,\zeta)+\rho(\zeta)}\right)^{N+k}\right.\\
\left.\prod_{s=1}^{k}\left(\rho(\zeta)\partial_{\overline{\zeta}_{a_{s}}}Q_{b_{s}}(z,\zeta)+\left(\partial\rho/\partial\overline{\zeta}_{a_{s}}\right)Q_{b_{s}}(z,\zeta)\right)\widetilde{f}\right).\label{eq:derivatrives-of-Tk(f)}
\end{multline}
where $\widetilde{f}$ is a coefficient of $\chi_{l}f$.

\subsubsection{Proof of Theorem \ref{thm:Theorem-Sobolev-Lp-k}}

We assume $N>2m(n+1)+5n+2\left|\gamma\right|+N_{0}(1)+1$. By \eqref{eq:derivatrives-of-Tk(f)}
we have to estimate the operators associated to the integrals
\begin{multline*}
I_{1,k}^{l}(z)=\int_{\Omega}\frac{z_{t}-\zeta_{t}}{\left|z-\zeta\right|^{2(n-k)}}\\
\Delta\left(\rho(\zeta)^{N-k}\left(\frac{1}{\frac{1}{K_{0}}S(z,\zeta)+\rho(\zeta)}\right)^{N+k}\right.\\
\left.\prod_{s=1}^{k}\left(\rho(\zeta)\partial_{\overline{\zeta}_{a_{s}}}Q_{b_{s}}(z,\zeta)+\left(\partial\rho/\partial\overline{\zeta}_{a_{s}}\right)Q_{b_{s}}(z,\zeta)\right)\widetilde{f}(\zeta)\right)d\lambda(\zeta)
\end{multline*}
where $\Delta=\frac{\partial}{\partial z_{j}}+\frac{\partial}{\partial\zeta_{j}}$.

If $\Delta$ acts on $\widetilde{f}$ the integral corresponds to
an operator treated in Section \ref{sec:Proof-of-Theorem-2-1}, so
we assume that $\Delta$ does not act on $\widetilde{f}$.

By a straightforward calculus we are lead to consider the integral
operators of the following type:
\begin{multline}
I=\int_{\Omega}*\frac{z_{t}-\zeta_{t}}{\left|z-\zeta\right|^{2(n-k)}}\\
\left(\rho(\zeta)^{N'-k}\left(\frac{1}{\frac{1}{K_{0}}S(z,\zeta)+\rho(\zeta)}\right)^{N'+k+1}\right.\\
\left.\prod_{s=1}^{k}\left(\rho(\zeta)\partial_{\overline{\zeta}_{a_{s}}}Q_{b_{s}}(z,\zeta)+\left(\partial\rho/\partial\overline{\zeta}_{a_{s}}\right)Q_{b_{s}}(z,\zeta)\right)\right)\widetilde{f}(\zeta)d\lambda(\zeta)\label{eq:one-derivative-1}
\end{multline}
where $N'=N\mathrm{\,or\,}N-1$ and $*$ is a smooth function or
\begin{multline}
I=\int_{\Omega}\frac{z_{t}-\zeta_{t}}{\left|z-\zeta\right|^{2(n-k)}}\\
\left(\rho(\zeta)^{N-k}\left(\frac{1}{\frac{1}{K_{0}}S(z,\zeta)+\rho(\zeta)}\right)^{N+k}\right.\\
\left.\Delta\prod_{s=1}^{k}\left(\rho(\zeta)\partial_{\overline{\zeta}_{a_{s}}}Q_{b_{s}}(z,\zeta)+\left(\partial\rho/\partial\overline{\zeta}_{a_{s}}\right)Q_{b_{s}}(z,\zeta)\right)\right)\widetilde{f}(\zeta)d\lambda(\zeta).\label{eq:one-derivative-2}
\end{multline}

If $l=0$ the kernels in $I$ are, by Lemma \ref{lem:pointwise-estimate-kernel-Kk}
of the form $\frac{*}{\left|z-\zeta\right|^{2(n-k)-1}}$, and, if
$l\geq1$ and $z\notin2V_{l}$ the kernels are smooth. Thus the integrals
corresponding are easy to estimate so we consider only the cases $l\geq1$,
$z\in2V_{l}$ and $k\leq n-r$ to study $I$ because of the holomorphic
property of the $Q_{i}$.

The following Lemma is elementary:
\begin{sllem}
\label{lem:weight-est-Lp1}Let $p\in[1,+\infty[$ and $\gamma>-1$.
There exists a constant $C>0$ such that, for any function $g\in\mathcal{C}^{1}(\Omega)$,
$\left\Vert g\right\Vert _{L^{p}\left(\delta_{\Omega}^{\nu}d\lambda\right)}\leq C_{\nu}\left\Vert g\right\Vert _{L_{1}^{p}\left(\delta_{\Omega}^{\gamma}d\lambda\right)}$
where $\delta_{\Omega}$ denotes the distance to the boundary of $\Omega$
and
\[
\nu>\left\{ \begin{array}{cc}
-1 & \text{if }\gamma<p-1,\\
\gamma-p & \text{if }\gamma\geq p-1.
\end{array}\right.
\]
\end{sllem}

\begin{proof}[Proof of (1) of Theorem \ref{thm:Theorem-Sobolev-Lp-k}]
 By Lemma \ref{lem:estimates-derivatives-Qi-in-kernel} the factors
of $\widetilde{f}$ in \eqref{eq:one-derivative-1} and \eqref{eq:one-derivative-2}
are bounded by
\[
\frac{1}{\left|z-\zeta\right|^{2(n-k)-1}}\left(\frac{\rho(\zeta)}{\frac{1}{K_{0}}S(z,\zeta)+\rho(\zeta)}\right)^{M}\frac{1}{\left|\rho(\zeta)\right|}\prod_{i=1}^{k}\tau_{i}\left(z,\delta_{\Omega}(z)+d(z,\zeta)\right)^{-2}.
\]

By Lemma \ref{lem:weight-est-Lp1} $\delta_{\Omega}(\zeta)f_{p}\in L^{p}\left(\delta_{\Omega}^{\gamma+\varepsilon}\right)$,
for all $\varepsilon>0$, and we can apply the last result of the
proof of Theorem \ref{thm:Theroem-without-derivative} to conclude.
\end{proof}

\begin{proof}[Proof of (2) of Theorem \ref{thm:Theorem-Sobolev-Lp-k}]
If $\gamma<p-1$ we use a transformation similar to one used by Deyun
Wu in \cite{wu-convex-98} for the integrals with $k\geq1$:
\begin{sllem}
\label{lem:Normal-derivative-weight-Wu}Let $\eta_{z}$ be the complex
normal at the point $z$ and let $\eta_{\zeta}^{*}=\frac{\partial}{\partial\eta_{z}}-\frac{\partial}{\partial\overline{\eta_{z}}}$
acting on the variable $\zeta$. Then for $\zeta\in V_{l}$ and $z\in2V_{l}$:
\begin{enumerate}
\item $\left|\eta_{\zeta}^{*}\rho(\zeta)\right|\apprle\left|z-\zeta\right|\apprle d(z,\zeta)^{1/m}\apprle\left|\frac{1}{K_{0}}S(z,\zeta)+\rho(\zeta)\right|^{1/m}$
\item $\left|\eta_{\zeta}^{*}S(z,\zeta)\right|\gtrsim1$
\item $\frac{1}{\left(\frac{1}{K_{0}}S(z,\zeta)+\rho(\zeta)\right)^{M}}=*\eta_{\zeta}^{*\alpha}\left(\frac{1}{\left(\frac{1}{K_{0}}S(z,\zeta)+\rho(\zeta)\right)^{M-\alpha}}\right)$
where $*$ is smooth in $\overline{\Omega}$.
\end{enumerate}
\end{sllem}

\begin{proof}
$\rho(\zeta)$ and $S(z,\zeta)$ are $\mathcal{C}^{\infty}$ functions
in the variable $\zeta$ on $\overline{\Omega}$ and $\eta_{\zeta}^{*}\rho($$z)=0$,
and, by the definition of the support function $S_{0}$, $\eta_{\zeta}^{*}S(z,z)\gtrsim1$
so (2) is valid if the balls $V_{l}$ are sufficiently small and,
for (1) we use \eqref{eq:min_denom_weight_bd+dist}.
\end{proof}
The two integrals \eqref{eq:one-derivative-1} and \eqref{eq:one-derivative-2}
are treated similarly so we only treat the first one.

Applying the preceding Lemma the integral \eqref{eq:one-derivative-1}
becomes
\begin{multline*}
I=\int_{\Omega}*\frac{z_{l}-\zeta_{l}}{\left|z-\zeta\right|^{2(n-k)}}\rho(\zeta)^{N'-k}\prod_{s=1}^{k}\left(\rho(\zeta)\partial_{\overline{\zeta}_{a_{s}}}Q_{b_{s}}(z,\zeta)+\left(\partial\rho/\partial\overline{\zeta}_{a_{s}}\right)Q_{b_{s}}(z,\zeta)\right)\\
\eta_{\zeta}^{*}\left(\left(\frac{1}{\frac{1}{K_{0}}S(z,\zeta)+\rho(\zeta)}\right)^{N'+k}\right)\widetilde{f}(\zeta)d\lambda(\zeta).
\end{multline*}

Now we integrate by parts the derivative $\eta_{\zeta}^{*}$. As the
integrant is identically zero on the boundary of $\Omega\cap V_{l}$
we get
\[
I=\int_{\Omega}K_{1}^{1}(z,\zeta)\widetilde{f}(\zeta)d\lambda(\zeta)+\int_{\Omega}K_{1}^{2}(z,\zeta)\eta_{\zeta}^{*}\widetilde{f}(\zeta)d\lambda(\zeta)=I_{1}^{1}+I_{1}^{2}
\]
and iterating this procedure for $I_{1}^{1}$ (which is possible because
$2(n-k)-1\leq2n-3$)
\[
I_{1}^{1}=\int_{\Omega}K_{2}^{1}(z,\zeta)\widetilde{f}(\zeta)d\lambda(\zeta)+\int_{\Omega}K_{2}^{2}(z,\zeta)\eta_{\zeta}^{*}\widetilde{f}(\zeta)d\lambda(\zeta)=I_{2}^{1}+I_{2}^{2}.
\]

If possible (integrability of a derivative of $\frac{z_{l}-\zeta_{l}}{\left|z-\zeta\right|^{2(n-k)}}$)
we iterate this procedure to $I_{2}^{1}$ and so on. Finally we get
a family of integrals
\[
I_{\alpha}^{1}=\int_{\Omega}K_{\alpha}^{1}(z,\zeta)\widetilde{f}d\lambda(\zeta)
\]
for $\alpha=\left(\alpha_{1},\alpha_{2},\alpha_{3}\right)$
\begin{multline*}
K_{\alpha}^{1}(z,\zeta)=*\left(\eta_{\zeta}^{*}\right)^{\alpha_{1}}\left(\frac{z_{l}-\zeta_{l}}{\left|z-\zeta\right|^{2(n-k)}}\right)\left(\eta_{\zeta}^{*}\right)^{\alpha_{2}}\left(\rho(\zeta)^{N'-k}\right)\left(\eta_{\zeta}^{*}\right)^{\alpha_{3}}\\
\left(\prod_{s=1}^{k}\left(\rho(\zeta)\partial_{\overline{\zeta}_{a_{s}}}Q_{b_{s}}(z,\zeta)+\left(\partial\rho/\partial\overline{\zeta}_{a_{s}}\right)Q_{b_{s}}(z,\zeta)\right)\right)\\
\left(\frac{1}{\frac{1}{K_{0}}S(z,\zeta)+\rho(\zeta)}\right)^{N'+k+1-|\alpha|},
\end{multline*}
the procedure being stopped when $\alpha_{1}=2k$ or $\alpha_{2}$
or $\alpha_{3}$ ``great enough'' (i. e. $\alpha_{2}=2m(n+1)$,
$\alpha_{3}=3n-1$).

By Lemma \ref{lem:estimates-derivatives-Qi-in-kernel}
\[
\left|\left(\eta_{\zeta}^{*}\right)^{\beta}\left(\rho(\zeta)\partial_{\overline{\zeta}_{a_{s}}}Q_{b_{s}}(z,\zeta)+\left(\partial\rho/\partial\overline{\zeta}_{a_{s}}\right)Q_{b_{s}}(z,\zeta)\right)\right|\apprle\left\{ \begin{array}{cc}
\frac{\epsilon^{2}}{\tau_{a_{s}}(\zeta,\epsilon)\tau_{b_{s}}(\zeta,\epsilon)} & \mathrm{\,if\,}\beta=0\\
\frac{\epsilon}{\tau_{a_{s}}(\zeta,\epsilon)\tau_{b_{s}}(\zeta,\epsilon)} & \mathrm{\,if\,}\beta=1\\
1 & \mathrm{\,if\,}\beta\geq2
\end{array}\right.
\]
where $\epsilon=\delta_{\Omega}(\zeta)+d(z,\zeta)$ which implies,
using that the $\left\{ a_{s}\right\} $ and $\left\{ b_{s}\right\} $
are distinct and the increasing property of the $\tau_{i}$,
\[
\left|\left(\eta_{\zeta}^{*}\right)^{\alpha_{3}}\left(\prod_{s=1}^{k}\left(\rho(\zeta)\partial_{\overline{\zeta}_{a_{s}}}Q_{b_{s}}(z,\zeta)+\left(\partial\rho/\partial\overline{\zeta}_{a_{s}}\right)Q_{b_{s}}(z,\zeta)\right)\right)\right|\apprle\epsilon^{2k-\alpha'_{3}}\prod_{i=1}^{k}\tau_{i}(\zeta,\epsilon)^{-2}
\]
where $\alpha'_{3}=\left\{ \begin{array}{cc}
\alpha_{3} & \mathrm{\,if\,}\alpha_{3}\leq2k\\
2k & \mathrm{\,if\,}\alpha_{3}\geq2k
\end{array}\right.$.

By (1) of Lemma \ref{lem:Normal-derivative-weight-Wu}
\[
\left|\left(\eta_{\zeta}^{*}\right)^{\alpha_{2}}\left(\rho(\zeta)^{N'-k}\right)\right|\apprle\left|\rho(\zeta)\right|^{N'-k-\alpha_{2}}A\left(\alpha_{2}\right)
\]
where
\begin{eqnarray*}
A\left(\alpha_{2}\right) & \apprle & \left|\eta_{\zeta}^{*}\rho(\zeta)\right|^{\alpha_{2}}+\left|\rho(\zeta)^{\frac{\alpha_{2}-1}{2}}\right|\\
 & \apprle & \left|\frac{1}{K_{0}}S(z,\zeta)+\rho(\zeta)\right|^{\frac{\alpha_{2}}{2m}}.
\end{eqnarray*}

Finally
\begin{multline*}
\left|K_{\alpha}^{1}(z,\zeta)\right|\apprle\frac{1}{\left|z-\zeta\right|^{2(n-k)-1+\alpha_{1}}}\left|\frac{\rho(\zeta)}{\frac{1}{K_{0}}S(z,\zeta)+\rho(\zeta)}\right|^{N'-k-\alpha_{2}}\\
\left|\frac{1}{K_{0}}S(z,\zeta)+\rho(\zeta)\right|^{\alpha_{1}-1+\alpha_{3}-\alpha'_{3}+\frac{\alpha_{2}}{2m}}\prod_{i=1}^{k}\tau_{i}(\zeta,\epsilon)^{-2}.
\end{multline*}

Note first that
\[
\left|K_{\alpha}^{1}(z,\zeta)\right|\apprle\frac{1}{\left|z-\zeta\right|^{2(n-k)-1+\alpha_{1}}}\left|\rho(\zeta)\right|^{\alpha_{1}-1}\left|\frac{\rho(\zeta)}{\frac{1}{K_{0}}S(z,\zeta)+\rho(\zeta)}\right|^{M}\prod_{i=1}^{k}\tau_{i}(\zeta,\epsilon)^{-2}
\]
for $M=N'-k-\alpha_{2}-\alpha_{1}+1>2\left|\gamma\right|+1$.

On the other hand if $\alpha_{2}$ or $\alpha_{3}$ is large then,
by Lemma \ref{lem:pointwise-estimate-kernel-Kk},
\[
\left|K_{\alpha}^{1}(z,\zeta)\right|\apprle\frac{1}{\left|z-\zeta\right|^{2n-1}}\left|\frac{\rho(\zeta)}{\frac{1}{K_{0}}S(z,\zeta)+\rho(\zeta)}\right|^{M}.
\]

If $\alpha_{1}=2k$ then, using \eqref{eq:estimation-tau-i-1-n},
\begin{eqnarray*}
\left|K_{\alpha}^{1}(z,\zeta)\right| & \apprle & \frac{1}{\left|z-\zeta\right|^{2n-1}}\left|\frac{1}{K_{0}}S(z,\zeta)+\rho(\zeta)\right|^{2k-1}\left|\frac{\rho(\zeta)}{\frac{1}{K_{0}}S(z,\zeta)+\rho(\zeta)}\right|^{M}\prod_{i=1}^{k}\tau_{i}(\zeta,\epsilon)^{-2}\\
 & \apprle & \frac{1}{\left|z-\zeta\right|^{2n-1}}\frac{1}{\left|\frac{1}{K_{0}}S(z,\zeta)+\rho(\zeta)\right|}\left|\frac{\rho(\zeta)}{\frac{1}{K_{0}}S(z,\zeta)+\rho(\zeta)}\right|^{M}=:W_{1}(z,\zeta).
\end{eqnarray*}

Now $K_{\alpha}^{2}(z,\zeta)=*K_{\alpha'}^{1}(z,\zeta)\left(\frac{1}{K_{0}}S(z,\zeta)+\rho(\zeta)\right)$
with $\left|\alpha'\right|=\left|\alpha\right|-1$, so
\[
\left|K_{\alpha}^{2}(z,\zeta)\right|\apprle\frac{1}{\left|z-\zeta\right|^{2(n-k)-1+\alpha_{1}}}\left|\rho(\zeta)\right|^{\alpha_{1}}\left|\frac{\rho(\zeta)}{\frac{1}{K_{0}}S(z,\zeta)+\rho(\zeta)}\right|^{M}\prod_{i=1}^{k}\tau_{i}(\zeta,\epsilon)^{-2}.
\]

If $\alpha_{1}$ is even ($\alpha_{1}=2\left(k-k_{1}\right)$), using
\eqref{eq:estimation-tau-i-1-n},
\begin{multline*}
\left|K_{\alpha}^{2}(z,\zeta)\right|\apprle\\
\frac{1}{\left|z-\zeta\right|^{2(n-k_{1})-1}}\left|\rho(\zeta)\right|^{2\left(k-k_{1}\right)}\left|\frac{\rho(\zeta)}{\frac{1}{K_{0}}S(z,\zeta)+\rho(\zeta)}\right|^{M}\prod_{i=1}^{k_{1}}\tau_{i}(\zeta,\epsilon)^{-2}\prod_{i=k_{1}+1}^{k}\tau_{i}(\zeta,\epsilon)^{-2}\\
\apprle\frac{1}{\left|z-\zeta\right|^{2(n-k_{1})-1}}\left|\frac{\rho(\zeta)}{\frac{1}{K_{0}}S(z,\zeta)+\rho(\zeta)}\right|^{M}\prod_{i=1}^{k_{1}}\tau_{i}(\zeta,\epsilon)^{-2}=:W_{2}(z,\zeta).
\end{multline*}

If $\alpha_{1}$ is odd then $\left|K_{\alpha}^{2}(z,\zeta)\right|\apprle\left|K_{\alpha'}^{2}(z,\zeta)\right|+\left|K_{\alpha''}^{2}(z,\zeta)\right|$
with $\alpha'=\left(\alpha_{1}-1,\alpha_{2},\alpha_{3}\right)$ and
$\alpha''=\left(\alpha_{1}+1,\alpha_{2},\alpha_{3}\right)$ so there
exists $k_{1}'$ and $k_{1}''$ such that
\begin{multline*}
\left|K_{\alpha}^{2}(z,\zeta)\right|\apprle\frac{1}{\left|z-\zeta\right|^{2(n-k_{1}')-1}}\prod_{i=1}^{k_{1}'}\tau_{i}(\zeta,\epsilon)^{-2}+\frac{1}{\left|z-\zeta\right|^{2(n-k_{1}'')-1}}\\
\left|\frac{\rho(\zeta)}{\frac{1}{K_{0}}S(z,\zeta)+\rho(\zeta)}\right|^{M}\prod_{i=1}^{k_{1}''}\tau_{i}(\zeta,\epsilon)^{-2}.
\end{multline*}

Finally $I$ is a sum of integrals applied to $\widetilde{f}$ with
kernels controlled by $W_{1}(z,\zeta)$ and applied to $\eta_{\zeta}^{*}\widetilde{f}$
with kernels controlled by $W_{2}(z,\zeta)$.

\medskip{}

Theorem \ref{thm:Theroem-without-derivative} has been proved applying
Lemma \ref{lem:kernel-operator-weighted-estimate-Lp-Lq}. If the conditions
of Lemma \ref{lem:kernel-operator-weighted-estimate-Lp-Lq}, for parameter
$s=1+\mu$, are verified then the operator $T$ maps $L^{p}\left(\delta^{\gamma}\right)$
into $L^{q}\left(\delta^{\gamma'}\right)$ for $\frac{1}{q}=\frac{1}{p}-\frac{\mu}{1+\mu}$.

For all $p\geq1$ under hypothesis of Theorem \ref{thm:Theroem-without-derivative}
we have proved the result if
\[
\mu\in\left[0,\frac{1-\frac{m}{p}(\gamma-\gamma')}{m(\gamma'+n)+2-(m-2)(r-1)}\right[=\left[0,\mu_{o}\right[
\]
(note that $\mu_{0}<\frac{1}{2n-1}$).

The integrals applied to $\eta_{\zeta}^{*}\widetilde{f}$ have kernels
bounded by those appearing in the proof of Theorem \ref{thm:Theroem-without-derivative}
and are well controlled when $\mu<\mu_{0}$.

For the integrals involving only $\widetilde{f}$ we have to estimate
the operator $T$ given by $Tf(z)=\int_{\Omega}K(z,\zeta)f(\zeta)d\lambda(\zeta)$
with
\begin{equation}
K(z,\zeta)=\frac{1}{\left|z-\zeta\right|^{2n-1}}\left[\frac{\rho(\zeta)}{\rho(\zeta)+\frac{1}{K_{0}}S(z,\zeta)}\right]^{M}\frac{1}{\left(\rho(\zeta)+\frac{1}{K_{0}}S(z,\zeta)\right)}\chi_{l}(\zeta).\label{eq:kernel-one-derivative}
\end{equation}
\end{proof}
Writing 
\[
Tf(z)=\int_{\Omega}K(z,\zeta)\delta(\zeta)^{-\alpha-\gamma'}\left(\delta^{\alpha}(\zeta)f(\zeta)\right)d\mu(\zeta)
\]
 with $d\mu(\zeta)=\delta^{\gamma'}(\zeta)d\lambda(\zeta)$ and $\alpha=\frac{\gamma-\gamma'-\beta}{p}$,
by Lemma \ref{lem:weight-est-Lp1} $\left\Vert \delta^{\alpha}f\right\Vert _{L^{p}\left(\delta^{\gamma'}\right)}\leq C_{\alpha}\left\Vert f\right\Vert _{L_{1}^{p}\left(\delta^{\gamma}\right)}$
with
\begin{equation}
\gamma-\gamma'<\beta<\gamma+1.\label{eq:conditions-on-beta}
\end{equation}
and it suffices to prove that the operator $T_{\beta}g(z)=\int_{\Omega}K(z,\zeta)\delta(\zeta)^{-\alpha-\gamma'}\left|g(\zeta)\right|d\mu(\zeta)$
maps $L^{p}\left(\delta^{\gamma'}\right)$ into $L^{q}\left(\delta^{\gamma'}\right)$
under the conditions on $\beta$ given in \eqref{eq:conditions-on-beta}.
\begin{itemize}
\item Case $p>1$. By Lemma \ref{lem:kernel-operator-weighted-estimate-Lp-Lq}
this result will be obtained proving the two following inequalities
\begin{equation}
I_{1}(z)=\int_{\Omega}\left|K(z,\zeta)\right|^{1+\mu}\delta(\zeta)^{\left(-\alpha-\gamma'\right)\left(1+\mu\right)+\gamma'-\varepsilon}d\lambda(\zeta)\apprle_{\varepsilon}\delta(z)^{-\varepsilon}\label{eq:I1-for-one-derivative}
\end{equation}
and
\begin{equation}
I_{2}(\zeta)=\int_{\Omega}\left|K(z,\zeta)\right|^{1+\mu}\delta(\zeta)^{\left(-\alpha-\gamma'\right)\left(1+\mu\right)}\delta(z)^{\gamma'-\varepsilon}d\lambda(z)\apprle_{\varepsilon}\delta(\zeta)^{-\varepsilon},\label{eq:I2-for-one-derivative}
\end{equation}
 with $\mu<\min\left(\frac{\beta-\left(\gamma-\gamma'\right)}{2np-\left(\beta-\left(\gamma-\gamma'\right)\right)+p\gamma'},\mu_{0}\right)=:\mu_{1}(\beta)$.\\
Let us verify first $I_{1}(z)$. If $l=0$, 
\[
\left|K(z,\zeta)\right|^{1+\mu}\delta(\zeta)^{\left(-\alpha-\gamma'\right)\left(1+\mu\right)+\gamma'-\varepsilon}\apprle\frac{1}{\left|z-\zeta\right|^{(2n-1)(1+\mu)}}
\]
and the result is trivial as $\mu<1/2n-1$.\\
If $l>0$: if $z\in\Omega\setminus2V_{l}$ the inequality is trivial
for similar reasons, so, we show it if $z\in2V_{l}$.\\
By \eqref{eq:kernel-one-derivative} and \eqref{eq:min_denom_weight_bd+dist}
\begin{multline*}
\left|K(z,\zeta)\right|^{1+\mu}\delta(\zeta)^{\left(-\alpha-\gamma'\right)\left(1+\mu\right)+\gamma'-\varepsilon}\leq\\
\frac{1}{\left|z-\zeta\right|^{(2n-1)(1+\mu)}}\frac{\left|\rho(\zeta)\right|^{M(1+\mu)}}{\left|\rho(\zeta)+\frac{1}{K_{0}}S(z,\zeta)\right|^{M(1+\mu)+1+\mu}}\delta(\zeta)^{\left(-\alpha-\gamma'\right)\left(1+\mu\right)+\gamma'-\varepsilon}\\
\apprle\frac{1}{\left|z-\zeta\right|^{(2n-1)(1+\mu)}}\left(\frac{1}{\delta(\zeta)+d(z,\zeta)}\right)^{-\gamma'+\left(\alpha+\gamma'\right)(1+\mu)+1+\mu+\varepsilon},
\end{multline*}
thus, if $\zeta\in P_{k}$ (recall that $P_{k}(z)=P\left(z,2^{k}\delta(z)\right)\setminus P\left(z,2^{k-1}\delta(z)\right)$
if $k\geq1$ and $P_{0}(z)=P\left(z,\delta(z)\right)$ and that the
hypothesis on $\mu$ imply $0\leq(2n-1)\mu<1$)
\begin{multline}
\int_{P_{k}(z)}\left|K(z,\zeta)\right|^{1+\mu}\delta(\zeta)^{\left(-\alpha-\gamma'\right)\left(1+\mu\right)+\gamma'-\varepsilon}\leq\\
\left(2^{k}\delta(z)\right)^{\gamma'-\left(\alpha+\gamma'\right)(1+\mu)-1-\mu-\varepsilon}\int_{P_{k}(z)}\frac{1}{\left|z-\zeta\right|^{(2n-1)(1+\mu)}}\\
\leq\left(2^{k}\delta(z)\right)^{1-(2n-1)\mu+\gamma'-\left(\alpha+\gamma'\right)(1+\mu)-1-\mu-\varepsilon}\text{ , by }\eqref{eq:integral-mod(z-zeta)-inverse-power-1+mu}.\label{eq:est_kernel_after_5_10}
\end{multline}
\\
Then $I_{1}(z)\apprle_{\mu}\delta_{\Omega}(z)^{-\varepsilon}$ if
$0\leq\mu<\mu_{1}$.\\
For $I_{2}(\zeta)$, using that $\delta(z)\apprle2^{k}\delta(\zeta)$
if $z\in P_{k}(\zeta)$ and \eqref{eq:3.16} the conclusion follows
as before if $\mu<\mu_{1}$.\\
If $\mu<\frac{\gamma'+1}{2np-(\gamma'+1)+p\gamma'}$ there exists
$\beta$, $\gamma'<\beta<\gamma+1$, such that
\[
\mu<\mu_{1}(\beta)
\]
achieving the proof for $p>1$, the case $k=0$ being trivial.
\item Case $p=1$. Formula \eqref{eq:est_kernel_after_5_10} and Lemma \ref{lem:kernel-operator-weighted-estimate-Lp-Lq}
imply the result.
\end{itemize}

\subsubsection{\label{sec:Proof-of-Theorem-2-3}Proof of Theorem \ref{thm:Theorem-Sobolev-Lp-k-0-compact-supp}}

Now we assume $N>2\left|\gamma\right|+d+2m(n+1)+5n+N_{0}(d)+1$. To
prove the theorem it suffices to prove that there exists a constant
$C$ depending only on $\Omega$ and $N$ such that for all $\left(0,r\right)$-form
$f$ with coefficients in $\mathcal{C}^{\infty}\left(\overline{\Omega}\right)$
with condition (1) and in $\mathcal{D}(\Omega)$ with condition (2)
\[
\left\Vert T_{K}(f)\right\Vert _{L_{k,r-1}^{q}\left(\delta_{\Omega}^{\gamma'}\right)}\leq C\left\Vert f\right\Vert _{L_{k,r}^{p}\left(\delta_{\Omega}^{\gamma}\right)},
\]
with $T_{K}(f)$ given by \eqref{eq:operator-Tk} and $p$, $\gamma$,
$q$ and $\gamma'$ as in the theorem.

\medskip{}

First we have to calculate the derivatives of the $T_{k}f$. As usual
we assume $l>0$ and $z\in2V_{l}$.

Using \eqref{eq:derivatrives-of-Tk(f)} for $J\geq1$, we have 
\[
\Delta_{J}\left(\left(\frac{\rho(\zeta)}{\frac{1}{K_{0}}S(z,\zeta)+\rho(\zeta)}\right)^{N+k}\prod_{s=1}^{k}\left(\frac{\partial_{\overline{\zeta}_{a_{s}}}Q_{b_{s}}(z,\zeta)}{\rho(\zeta)}+\frac{\left(\partial\rho/\partial\overline{\zeta}_{a_{s}}\right)Q_{b_{s}}(z,\zeta)}{\rho^{2}(\zeta)}\right)\widetilde{f}\right)
\]
 is a sum of expressions of the form
\begin{multline*}
\Delta_{J_{1}}\left(\left(\frac{\rho(\zeta)}{\frac{1}{K_{0}}S(z,\zeta)+\rho(\zeta)}\right)^{N+k}\right)\\
\Delta_{J_{2}}\left(\prod_{s=1}^{k}\left(\frac{\partial_{\overline{\zeta}_{a_{s}}}Q_{b_{s}}(z,\zeta)}{\rho(\zeta)}+\frac{\left(\partial\rho/\partial\overline{\zeta}_{a_{s}}\right)Q_{b_{s}}(z,\zeta)}{\rho^{2}(\zeta)}\right)\right)\Delta_{J_{3}}\widetilde{f}
\end{multline*}
where $\Delta_{J_{t}}=\BCo_{j=1}^{J_{t}}\left(\frac{\partial}{\partial z_{i_{j}}}+\frac{\partial}{\partial\zeta_{i_{j}}}\right)$
with $J_{1}+J_{2}+J_{3}=J$.

\begin{spprop}
\label{prop:final-estimate-kernel-J-derivated}Let $D_{J}^{z}=\BCo_{j=1}^{J}\frac{\partial}{\partial z_{i_{j}}}=\frac{\partial^{J}}{\partial z_{i_{j}}\ldots\partial z_{i_{J}}}$,
$i_{j}\in\left\{ 1,\ldots,n\right\} $ be a derivative of order $J$,
$l>0$. Then $D_{J}^{z}T_{K}f$ is controlled by a sum of integrals
$\int_{\Omega\cap2V_{l}}K_{i,l}(z,\zeta)D_{i}^{\zeta}\widetilde{f}(\zeta)$,
$0\leq i\leq J$, where $D_{i}^{\zeta}\widetilde{f}$ is a derivative
of $\widetilde{f}$ of order $J-i$ and
\begin{multline*}
\left|K_{i,l}(z,\zeta)\right|\apprle\left(\sum_{k=0}^{r-1}\frac{1}{\left|z-\zeta\right|^{2(n-k)-1}}\left[\frac{\rho(\zeta)}{\rho(\zeta)+\frac{1}{K_{0}}S(z,\zeta)}\right]^{M}\prod_{j=1}^{k}\tau_{j}\left(z,\alpha\right)^{-2}\right)\\
\frac{1}{\left|\rho(\zeta)+\frac{1}{K_{0}}S(z,\zeta)\right|^{i}}\chi_{l}(\zeta),
\end{multline*}
with $\alpha=\delta_{\Omega}(z)+d(z,\zeta)$ and $M$ large ($N$
being large depending on $J$) and $\delta_{\Omega}$ denotes the
distance to the boundary.
\end{spprop}

\begin{proof}
Comes from Lemma \ref{lem:estimates-derivatives-Qi-in-kernel} and
Lemma \ref{lem:pointwise-estimate-kernel-Kk}.
\end{proof}
\begin{sllem}
\label{lem:weight-est-Lpk-0}Let $p\in[1,+\infty[$, $\gamma\in\mathbb{R}$,
$d'\in\mathbb{N}$ and $g\in\mathcal{C}^{\infty}(\Omega)$. Then under
one of the two following condition
\begin{enumerate}
\item $\gamma\geq d'p-1$
\item $g$ has compact support in $\Omega$ (i.e. $g\in\mathcal{D}(\Omega)$)
\end{enumerate}
for $\nu>\gamma-d'p$ there exists a constant $C_{\nu}$ such that
$\left\Vert g\right\Vert _{L^{p}\left(\delta_{\Omega}^{\nu}d\lambda\right)}\leq C_{\nu}\left\Vert g\right\Vert _{L_{d'}^{p}\left(\delta_{\Omega}^{\gamma}d\lambda\right)}$.
\end{sllem}

\begin{proof}
Assume condition (1) satisfied. Then, by induction on $i=1,\ldots,d'$,
Lemma \ref{lem:weight-est-Lp1} gives $\left\Vert g\right\Vert _{L_{d'-i}^{p}\left(\delta_{\Omega}^{\nu}\right)}\leq C_{\nu}\left\Vert g\right\Vert _{L_{d'}^{p}\left(\delta_{\Omega}^{\gamma}\right)}$
for $\nu>\gamma-ip$.

Assume now condition (2) satisfied. Remark first that the following
Claim is an elementary consequence of Hölder inequality:
\begin{claim*}
Let $p\in[1,+\infty[$ and $\eta<p-1$. Then for $\nu>\eta-p$ and
any function $h\in\mathcal{D}(\Omega)$, $\left\Vert h\right\Vert _{L^{p}\left(\delta_{\Omega}^{\nu}d\lambda\right)}\leq C_{\nu}\left\Vert h\right\Vert _{L_{1}^{p}\left(\delta_{\Omega}^{\eta}d\lambda\right)}$.
\end{claim*}

Then using the Claim and Lemma \ref{lem:weight-est-Lp1} the Lemma
\ref{lem:weight-est-Lpk-0} is easily obtained by induction on $i=1,\ldots,d'$:
for $\nu>\gamma-ip$, $\left\Vert g\right\Vert _{L_{d'-i}^{p}\left(\delta_{\Omega}^{\nu}\right)}\leq C_{\nu}\left\Vert g\right\Vert _{L_{d'}^{p}\left(\delta_{\Omega}^{\gamma}\right)}$.
\end{proof}
\medskip{}

Let us now finish the proof of Theorem \ref{thm:Theorem-Sobolev-Lp-k-0-compact-supp}.

For the two first points of the theorem we use this Lemma (with $d'=d$)
and Proposition \ref{prop:final-estimate-kernel-J-derivated} (recall
$\left|\rho(\zeta)+\frac{1}{K_{0}}S(z,\zeta)\right|\apprge\delta_{\Omega}(\zeta)$).

For the last point we use the above Lemma (with $d'=d-1$) and Theorem
\ref{thm:Theorem-Sobolev-Lp-k}.

\section{\label{sec:Application-to-weighted-Bergman}Application to weighted
Bergman projections}

In view of theorems \ref{thm:Theorem-Sobolev-Lp-k} and \ref{thm:Theorem-Sobolev-Lp-k-0-compact-supp}
we have:
\begin{spprop}
\label{prop:epsilon-gains-for-solutions}Let $d\in\mathbb{N_{*}}$,
$p_{0}\geq1$, $\gamma_{0}$ $\gamma_{1}$, $(d-1)p_{0}-1<\gamma_{0}\leq\gamma_{1}$.
There exists an operator $T$ solving the $\overline{\partial}$-equation
and $\varepsilon>0$ such that for $p\in\left[1,p_{0}\right]$ and
$\gamma\in\left[\gamma_{0},\gamma_{1}\right]$:
\begin{enumerate}
\item $T$ maps continuously $L_{d}^{p}\left(\delta_{\Omega}^{\gamma}\right)$
in $L_{d}^{p+\varepsilon}\left(\delta_{\Omega}^{\gamma}\right)$;
\item $T$ maps continuously $L_{d}^{p}\left(\delta_{\Omega}^{\gamma}\right)$
in $L_{d}^{p}\left(\delta_{\Omega}^{\gamma-\varepsilon}\right)$.
\end{enumerate}
\end{spprop}

\begin{proof}
We use the operator $T$ defined in \eqref{eq:def-op-solv-d-bar}
with $N>2\left|\gamma_{1}\right|+2m(n+2)+5n+d+1$.
\end{proof}

The following Proposition is a special case of \cite[Theorem 2.1]{CPDY}:
\begin{spprop}
Let $D$ be a smoothly bounded convex domain of finite type in $\mathbb{C}^{n}$.
Let $\rho$ be a smooth defining function of $D$. Let $r\in\mathbb{R}_{+}$,
be a non negative real number, $\eta\in\mathcal{C}^{\infty}(\overline{D})$
be strictly positive and $\omega=\eta\left|\rho\right|^{r}$. Then,
for any integer $d$, the weighted Bergman projection $P_{\omega}^{\Omega}$
of the Hilbert space $L^{2}\left(\Omega,\omega d\lambda\right)$ maps
continuously the weighted Sobolev space $L_{d}^{2}\left(\Omega,\omega d\lambda\right)$
into itself.
\end{spprop}

In \cite[(1) of Theorem 1.1]{CDM} the following result was proved:
\begin{spprop}
Let $D$ be a smoothly bounded convex domain of finite type. Let $g$
the gauge of $D$, $\rho_{0}=g^{4}e^{1-\nicefrac{1}{g}}-1$ and
$\omega_{0}=\left(-\rho_{0}\right)^{r}$ with $r$ a non negative
rational number. Let $P_{\omega}$ the weighted Bergman projection
of $D$ associated to the Hilbert space $L^{2}\left(D,\omega_{0}d\lambda\right)$,
$d\lambda$ denoting the Lebesgue measure. Then for $s\in\mathbb{N}$,
$p\in\left]1,+\infty\right[$ and $-1<\beta<p\left(r+1\right)-1$,
$P_{\omega_{0}}$ maps the Sobolev space $L_{s}^{p}\left(D,\delta_{\partial\Omega}^{\beta}\right)$
continuously into itself.
\end{spprop}

The following Proposition extends partially this result:
\begin{spprop}
Let $d\in\mathbb{N}_{*}$. Let $D$ be a smoothly bounded convex domain
of finite type in $\mathbb{C}^{n}$. Let $\chi$ be any $\mathcal{C}^{\infty}$
non negative function in $\overline{D}$ which is equivalent to the
distance $\delta_{D}$ to the boundary of $D$ and let $\eta$ be
a strictly positive $\mathcal{C}^{\infty}$ function on $\overline{D}$.
Let $P_{\omega}$ be the (weighted) Bergman projection of the Hilbert
space $L^{2}\left(D,\omega d\lambda\right)$ where $\omega=\eta\chi^{r}$
with $r$ a non negative rational number. Then:

\begin{enumerate}
\item For $p\in\left[2,+\infty\right[$ and $-1<\beta\leq r$, $P_{\omega}$
maps the Sobolev space $L_{1}^{p}\left(D,\delta_{D}^{\beta}\right)$
continuously into itself.
\item If $d\geq2$, for $p\in\left[2,+\infty\right[$ if $r>dp-1$ and $d(p-1)-1<\beta\leq r$,
$P_{\omega}$ maps continuously the Sobolev space $L_{d}^{p}\left(D,\delta_{D}^{\beta}\right)$
into itself.
\end{enumerate}
\end{spprop}

Using Proposition \ref{prop:epsilon-gains-for-solutions} and the
two previous Propositions, the proof is completely similar to the
proof done in Section 4 of \cite{ChDu18} based on the following comparison
formula
\[
\varphi P_{\omega}(u)=P_{\omega_{0}}(\varphi u)+\left(\mathrm{Id}-P_{\omega_{0}}\right)\circ A\left(P_{\omega}(u)\wedge\overline{\partial}\varphi\right),
\]
where $\omega=\varphi\omega_{0}$ and $A$ is an operator solving
the $\overline{\partial}$-equation.

\bibliographystyle{amsalpha}
\providecommand{\bysame}{\leavevmode\hbox to3em{\hrulefill}\thinspace}
\providecommand{\MR}{\relax\ifhmode\unskip\space\fi MR }
\providecommand{\MRhref}[2]{%
  \href{http://www.ams.org/mathscinet-getitem?mr=#1}{#2}
}
\providecommand{\href}[2]{#2}

\end{document}